\documentclass[12pt,reqno]{amsart}
\usepackage{amsmath,amssymb,euscript}
\usepackage{enumitem} 

\usepackage{tikz}

\def\BB/{B\"acklund}
\def\BT/{B\"acklund transformation}
\def\MA/{Monge-Amp\`ere}

\def\R{{\mathbb R}}
\def\B{{\EuScript B}}
\def\M{{\EuScript M}}
\renewcommand{\P}{{\EuScript P}}

\def\I{{\mathcal I}}
\def\J{{\EuScript J}}

\def\w{\omega}
\def\vomega{\boldsymbol{\omega}}
\def\other#1{\underline{#1}}
\def\thetabar{\other{\theta}}
\def\Omegabar{\other{\Omega}}
\def\xbar{\other{x}}
\def\ybar{\other{y}}
\def\ubar{\other{u}}
\def\pbar{\other{p}}
\def\qbar{\other{q}}
\def\Lbar{\other{L}}
\def\&{\wedge}
\def\di{\partial}
\def\Mbar{\other{\M}}
\def\pibar{\other{\pi}}
\def\Ibar{\other{\I}}
\def\teta{\tilde{\eta}}
\def\talpha{\tilde{\alpha}}
\def\tbeta{\tilde{\beta}}
\def\tkappa{\tilde{\kappa}}

\def\hoch#1{\hat{#1}}
\def\hw{\hoch\w}
\def\htheta{\hoch\theta}
\def\hthetabar{\hoch\thetabar} 
\def\unter#1{{#1}}
\def\utheta{\unter\theta}
\def\uthetabar{\unter\thetabar}
\def\uw{\unter\w}
\def\Lie{\mathcal{L}}

\def\wt{\widetilde}

\def\Fp{F_1'}
\def\Fpp{F_1''}
\def\Gp{G_1'}
\def\Gpp{G_1''}
\newcommand\dibu{\dfrac{\di}{\di u}}
\newcommand\dibv{\dfrac{\di}{\di v}}
\newcommand\dibh{\dfrac{\di}{\di h}}
\newcommand\dibk{\dfrac{\di}{\di k}}

\newcommand\igenl{\left\langle}
\newcommand\igenr{\right\rangle}

\DeclareMathAccent{\widehat}{\mathord}{largesymbols}{"62}
\DeclareFontFamily{U}{mathx}{\hyphenchar\font45}
\DeclareFontShape{U}{mathx}{m}{n}{
      <5> <6> <7> <8> <9> <10>
      <10.95> <12> <14.4> <17.28> <20.74> <24.88>
      mathx10
      }{}
\DeclareSymbolFont{mathx}{U}{mathx}{m}{n}
\DeclareFontSubstitution{U}{mathx}{m}{n}
\DeclareMathAccent{\widecheck}{0}{mathx}{"71}

\newcommand{\intprod}{\mathbin{\raisebox{.4ex}{\hbox{\vrule height .5pt width
5pt depth 0pt %
         \vrule height 3pt width .5pt depth 0pt}}}}

\newtheoremstyle{mytheoremstyle} 
    {.5\headsep}                    
    {.5\headsep}                    
    {\sl}                   
    {}                           
    {\bf}                   
    {.}                          
    {.5em}                       
    {}  

\theoremstyle{mytheoremstyle}
\newtheorem{theorem}{Theorem}[section]
\newtheorem{lemma}[theorem]{Lemma}
\newtheorem{cor}[theorem]{Corollary}
\newtheorem{prop}[theorem]{Proposition}
\newtheorem{definition}[theorem]{Definition}
\newtheorem*{theorem*}{Theorem}

\theoremstyle{definition}
\newtheorem{remark}[theorem]{Remark}
\newtheorem*{remark*}{Remark}

\numberwithin{equation}{section}

\begin{document}
\title{Geometric characterization and classification of B\"acklund transformations of sine-Gordon type}
\author{Jeanne N. Clelland}
\address{Department of Mathematics, 395 UCB, University of
Colorado,
Boulder, CO 80309-0395}
\email{Jeanne.Clelland@colorado.edu}
\author{Thomas A. Ivey}
\address{Dept. of Mathematics, College of Charleston\\
66 George St., Charleston SC 29424-0001}
\email{IveyT@cofc.edu}
\date{\today}


\thanks{The first author was supported in part by NSF grants DMS-0908456 and DMS-1206272.}

\begin{abstract}
We begin by considering several properties commonly (but not universally) possessed by \BT/s between
 hyperbolic Monge-Amp\`ere equations: wavelike nature of the underlying equations,
 preservation of independent variables, quasilinearity of the transformation, and autonomy of the transformation.
We show that, while these properties all appear to depend on the formulation of both the underlying PDEs and the B\"acklund transformation in a particular
coordinate system, in fact they all have intrinsic geometric meaning, independent of any particular choice of local coordinates.

Next, we consider the problem of classifying
\BT/s with these properties. We show that, apart from a family of transformations between Monge-integrable equations, there exists only a finite-dimensional family of such transformations, including the well-known family of \BT/s for the sine-Gordon equation. The full extent of this family is not yet determined, but our analysis has uncovered previously unknown transformations among generalizations of Liouville's equation.
\end{abstract}

\maketitle

\section{Introduction}
The classical B\"acklund transformation for the sine-Gordon equation
\begin{equation}
z_{xy}=\sin z \label{SGE}
\end{equation}
is a system of two partial differential equations for unknown functions $u(x,y)$
and $v(x,y)$:
\begin{equation} \label{sgbt}
\begin{aligned}
u_x - v_x &= 2\lambda \sin\left(\frac{u+v}{2}\right), \\
u_y +v_y&= \frac{2}{\lambda}\sin\left(\frac{u-v}{2}\right).
\end{aligned}
\end{equation}
The system \eqref{sgbt} has the property that if $z = u(x,y)$ is a given solution of \eqref{SGE},
then solving the system
for $v(x,y)$ with $\lambda$ fixed
gives a 1-parameter family of new solutions $z = v(x,y)$ of \eqref{SGE}.
(For example, starting with the trivial solution $u=0$ gives the
1-soliton solutions of sine-Gordon, with initial position depending
on a constant of integration and velocity depending on the choice of the nonzero constant $\lambda$.)
Note that, given $u$,  the partial derivatives of $v$ are completely determined;
the compatibility condition of the two resulting equations for $v$ is precisely that $u$ satisfies the sine-Gordon equation \eqref{SGE}.

In general, \BT/s provide a way to obtain new solutions of a partial differential equation (or system of PDEs)
by starting with a given solution of the same (or a different) PDE and solving an auxiliary system
of {\em ordinary} differential equations.  However, the transformation \eqref{sgbt} has some special
properties, including:
\begin{enumerate}
\item\label{autoprop} the transformation \eqref{sgbt} is an {\em auto}-\BT/, i.e., it links two solutions of the {\em same} PDE---namely, the sine-Gordon equation \eqref{SGE};
\item\label{startprop} the underlying PDE \eqref{SGE} is a {\em wavelike equation}---i.e., a hyperbolic PDE of the form $z_{xy} = f(x,y,z,z_x,z_y)$;
\item the transformation \eqref{sgbt} preserves the independent variables $x$ and $y$;
\item the relations between the partial derivatives of $u$ and $v$ defined by the transformation \eqref{sgbt} are linear;
\item\label{endprop} the transformation \eqref{sgbt} has no explicit dependence on $x$ or $y$--thus, we say that the transformation \eqref{sgbt} is {\em autonomous};
\item\label{paramprop} the transformation \eqref{sgbt} is actually a one-parameter family of transformations, depending on $\lambda$.
\end{enumerate}

B\"acklund transformations with property (\ref{paramprop}) were the subject of one of our earlier papers \cite{CI05}.
In the present paper, we will concentrate on the other properties on this list.  We begin in \S \ref{review-sec} by reviewing the geometric formulation given in \cite{C02} for B\"acklund transformations of hyperbolic Monge-Amp\`ere systems in terms of exterior differential systems.  In \S \ref{results-sec} we show that, while properties (\ref{startprop})--(\ref{endprop}) all appear
to depend on the formulation of both the PDE \eqref{SGE} and the B\"acklund transformation \eqref{sgbt}
in a particular coordinate system, in fact they all have intrinsic geometric meaning, independent
of any particular choice of local coordinates.  In \S \ref{classification-sec}, we consider the problem of {\em classifying} \BT/s with these properties, and we give a characterization of quasilinear, autonomous, wavelike \BT/s as solutions of the overdetermined PDE system \eqref{bigsys}.  We perform a detailed analysis of the solution space of this system in order to classify such transformations; in particular, we show that apart from a family of transformations between Monge-integrable equations, the space of all such transformations is finite-dimensional.
Finally, in \S \ref{discussion-sec} we discuss some of the limitations of our approach.

Perhaps surprisingly, a geometric characterization of property (\ref{autoprop})--being an
auto-\BT/--in terms of invariants
for the associated exterior differential system remains elusive; we hope to consider this issue in a future paper.

\section{Geometric formulation of \BT/s}\label{review-sec}

In this section, we review how to formulate hyperbolic Monge-Amp\`ere PDEs, as well
as B\"acklund transformations between them, as exterior differential systems.
We will use the sine-Gordon equation \eqref{SGE} and its B\"acklund transformation
\eqref{sgbt} as examples to illustrate the general constructions.

\subsection{Hyperbolic Monge-Amp\`ere systems}

The existence of a \BT/ between two partial differential
equations is a property that is independent of changes of coordinates.  The geometric
viewpoint we adopt for studying such properties is that of {\em exterior differential systems},
in which a PDE or system of PDEs is described by a differentially closed ideal $\I$ of differential forms on
a manifold, and solutions to the PDE are in one-to-one correspondence with submanifolds to which
the forms in the ideal pull back to be zero.
(Such submanifolds are called {\em integral submanifolds} or {\em integrals} of the system.)

For example, if we let $\I$ be the differential ideal generated by the differential forms
\begin{equation}\label{siggy}
\theta = du - p\,dx -q\,dy, \qquad \Omega = (dp -(\sin u)\,dy)\wedge dx
\end{equation}
on the manifold $\R^5$ with coordinates $(x,y,u,p,q)$, then solutions of the sine-Gordon equation \eqref{SGE}
are in one-to-one correspondence with surfaces in $\R^5$ on which $\theta$, $\Omega$,
and their exterior derivatives vanish, and on which the 2-form $dx \wedge dy$ is never zero.
This can be seen as follows: the condition that $dx \wedge dy$ is nonvanishing on a surface
$\Sigma \subset \R^5$ is equivalent to the condition that $\Sigma$ is a graph over the the $xy$ plane.
Thus $\Sigma$ is defined by equations of the form
\[ u = u(x,y), \qquad p = p(x,y), \qquad q = q(x,y). \]
Then the vanishing of $\theta$ implies that $p=u_x$ and $q=u_y$, while the
vanishing of $\Omega$ implies that $u_{xy} = p_y = \sin u$.

This is an example of a {\em Monge-Amp\`ere} exterior differential system on a 5-dimensional manifold $\M$.
Any such exterior differential system $\I$ is generated locally by a contact 1-form $\theta$ and a
2-form $\Omega$, with the property that at each point
$\Omega$ is linearly independent from $d\theta$ and wedge products with $\theta$.
As an ideal within the ring of differential forms on $\M$, $\I$ is generated algebraically by $\theta$, $d\theta$, and $\Omega$;
we will denote this by $\I = \igenl \theta, d\theta, \Omega \igenr.$
Given such a system, the Pfaff theorem implies that there always exist local coordinates $(x,y,u,p,q)$ such that (up to a nonzero multiple)
\[ \theta=du-p \,dx -q\,dy. \]
Then by subtracting off suitable multiples of $\theta$ and $d\theta$, we can assume that
\[ \Omega = A\, dp \& dy + \tfrac{1}{2}B\, (dx \& dp + dq \& dy) + C\, dx \& dq + D \,dp \& dq + E\, dx \& dy \]
for some functions $A,B,C,D,E$.
Thus, by the same argument as above, integral surfaces of $\I$ on which $dx \wedge dy$ is never zero are in one-to-one correspondence
with solutions of a \MA/ PDE
\begin{equation}
A u_{xx} + B u_{xy} + C u_{yy} + D (u_{xx} u_{yy} - u_{xy}^2) + E = 0, \label{MA-PDE}
\end{equation}
where $A,B,C,D,E$ are functions of the variables $(x,y,u,u_x, u_y)$.

\MA/ equations comprise the smallest class of second-order PDEs for one function of two variables
that is invariant under contact transformations and contains the quasilinear equations (see, e.g., Chapter 2
in \cite{Goursat}).
So, if one is interested in
studying equations like the sine-Gordon equation \eqref{SGE} from a geometric viewpoint, it is natural to focus on \MA/ exterior differential systems.
Note that in the example \eqref{siggy},
the 2-form generator $\Omega$ is {\em decomposable}, i.e., a wedge-product of two 1-forms.
In fact, we can choose algebraic generators for $\I$ consisting of $\theta$ and two 2-forms $\Omega_1, \Omega_2$ that are both decomposable.  Specifically, if we take
\[ \Omega_1 = (dp -(\sin u)\,dy)\wedge dx, \qquad \Omega_2 = (dq -(\sin u)\,dx)\wedge dy, \]
then we have
\[ \Omega = \Omega_1, \qquad d\theta = -(\Omega_1 + \Omega_2), \]
and so
\[ \I = \igenl\theta, d\theta, \Omega\igenr = \igenl\theta, \Omega_1, \Omega_2\igenr. \]
\MA/ systems with this property are called {\em hyperbolic}, because this decomposability condition
is equivalent to the condition that the corresponding PDE \eqref{MA-PDE} is hyperbolic in the usual sense.

\subsection{B\"acklund transformations of hyperbolic \MA/ systems}
\label{BTdef}
The following geometric definition of a \BT/ between two hyperbolic \MA/ systems is based on that given in \cite{CI09} and \cite{CFB2} (see,
for example, Defn. 7.5.10 in the latter reference).

\begin{definition}
Let $(\M, \I)$, $(\Mbar, \Ibar)$ be hyperbolic Monge-Amp\`ere systems on 5-dimensional manifolds $\M, \Mbar$
respectively.  A {\em \BT/} between $(\M, \I)$ and $(\Mbar, \Ibar)$ is a 
manifold $\B$, equipped with submersions $\pi:\B \to \M$ and $\pibar:\B \to \Mbar$ whose fibers are transverse, and a Pfaffian exterior differential system $\J$
on $\B$ with the property that $\J$ is an integrable extension of both $\I$ and $\Ibar$.
\end{definition}

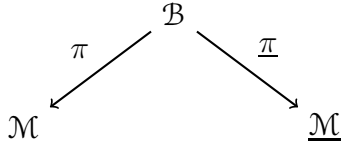
\begin{center}
\begin{tikzpicture}
\node (total) at (0,1.5) {$\B$};
\node (ujet) at (-2,0) {$\M$};
\node (vjet) at (2,0) {$\Mbar$};
\draw[->,thick] (total) -- node[above left]  {$\pi$} (ujet);
\draw[->,thick] (total) -- node[above right] {$\pibar$} (vjet);
\end{tikzpicture}
\end{center}

In this article, we will restrict our attention to the lowest-dimensional case, where $\B$ is a 6-dimensional manifold and $\J$ has rank 2.
(However, not all \BT/s of interest arise in this way; see \S\ref{discussion-sec}.)
Let
\[ \I = \igenl\theta, \Omega_1, \Omega_2\igenr, \qquad \Ibar = \igenl\thetabar, \Omegabar_1, \Omegabar_2\igenr \]
where the $\Omega_i$ and $\Omegabar_i$ are decomposable 2-forms.
Then $\J$ is the differential ideal on $\B$ generated by the 1-forms $\pi^*\theta, \pibar^*\thetabar$
and their exterior derivatives.  The condition that $\J$ be an integrable extension of both $\I$ and $\Ibar$ means that
\begin{equation}
\begin{aligned}
\pi^* d\theta &\equiv 0 \mod \pi^*\theta, \pibar^* \thetabar,  \pibar^*\Omegabar_1, \pibar^*\Omegabar_2,\\
\pibar^* d\thetabar &\equiv 0 \mod \pibar^*\thetabar, \pi^*\theta, \pi^* \Omega_1, \pi^* \Omega_2.\\
\end{aligned} \label{int-conds}
\end{equation}
Consequently, when $\J$ is restricted to the inverse image (under $\pi$) of an integral submanifold of $\I$, it
satisfies the Frobenius integrability condition, and similarly for the inverse image of an integral of $\Ibar$.
We will also assume that the \BT/ satisfies the technical condition that $\pi^* d\theta$ and
$\pibar^* d\thetabar$ are linearly independent modulo $\pi^*\theta$ and $\pibar^* \thetabar$; we
call such transformations {\em normal}.

For example, in the sine-Gordon example above, the systems $(\M, \I)$, $(\Mbar, \Ibar)$ are each taken to be
copies of the exterior differential system described in the previous section: $\M = \R^5$ with coordinates $(x,y,u,p,q)$,
$\Mbar = \R^5$ with coordinates $(\xbar, \ybar, \ubar, \pbar, \qbar)$, and
\begin{align*}
\I &= \igenl \theta = du - p\, dx - q\, dy,\ \Omega_1 = (dp -(\sin u)\,dy)\wedge dx, \ \Omega_2 = (dq -(\sin u)\,dx)\wedge dy \igenr, \\
\Ibar & = \igenl \thetabar = d\ubar - \pbar\, d\xbar - \qbar\, d\ybar,\ \Omegabar_1 = (d\pbar -(\sin \ubar)\,d\ybar)\wedge d\xbar, \ \Omegabar_2 = (d\qbar -(\sin \ubar)\,d\xbar)\wedge d\ybar \igenr.
\end{align*}
$\B$ is the 6-dimensional submanifold of $\M \times \Mbar$ defined by the equations $\xbar = x, \ybar = y$
(because the transformation preserves the independent variables $x,y$) and the two equations
\begin{equation}
 p - \pbar  = 2 \lambda \sin\left(\frac{u+\ubar}{2}\right), \qquad
q + \qbar  = \frac{2}{\lambda}\sin\left(\frac{u-\ubar}{2}\right), \label{SGEback-defineB}
\end{equation}
which are equivalent to equations \eqref{sgbt}.  (We will see later that it is advantageous to
regard these equations as defining $p, \qbar$ as functions of the independent variables $(x, y, u, \ubar, \pbar, q)$ on $\B$.)
It is straightforward to check that the pullbacks of $\I, \Ibar$ to $\B$ satisfy the integrability conditions \eqref{int-conds},
and that $(\B, \J)$ is a normal \BT/ between $(\M, \I)$ and $(\Mbar, \Ibar)$.
The fact that the restriction of $\J$ to the inverse image of an integral submanifold of $\I$
satisfies the Frobenius condition is equivalent to the statement that whenever the function $u(x,y)$
is a solution of the sine-Gordon equation \eqref{SGE}, the system \eqref{sgbt} is a compatible system
for the unknown function $v(x,y)$ whose solutions can be constructed by solving ODEs.

In \cite{C02} it is shown that a normal \BT/ between two hyperbolic \MA/ systems determines,
and is determined by, a $G$-structure $\P$ on the 6-dimensional manifold $\B$, where
$G\subset GL(6,\R)$ consists of matrices of the form
\begin{equation}\label{groupie}
\begin{pmatrix} a & 0 & 0 & 0\\ 0 & b & 0 & 0 \\ 0 &0 & B & 0 \\ 0 & 0 & 0 & A\end{pmatrix},
\qquad A,B \in GL(2,\R), \  a=\det A, b=\det B.
\end{equation}
(Recall that this means $\P$ is a principal sub-bundle of the general linear coframe bundle on $\B$, and a simple transitive $G$-action on the fibers
of $\P$ is induced by the inclusion $G \subset GL(6,\R)$.)
We will use the notation $(\htheta, \hthetabar, \hw^1, \hw^2, \hw^3, \hw^4)$ for the components
of the canonical $\R^6$-valued 1-form on $\P$.  In \cite{C02} is it shown that
there is a connection form $\Upsilon$ on $\P$ (taking value in the Lie algebra of $G$) such that the canonical forms satisfy structure equations
\begin{equation}\label{streq}
d\begin{bmatrix} \htheta \\ \hthetabar \\ \hw^1 \\ \hw^2 \\ \hw^3 \\ \hw^4 \end{bmatrix}
= \Upsilon \& \begin{bmatrix} \htheta \\ \hthetabar \\ \hw^1 \\ \hw^2 \\ \hw^3 \\ \hw^4 \end{bmatrix}
+ \begin{bmatrix}
A_1(\hw^1-C_1 \htheta) \& (\hw^2 -C_2\htheta) + \hw^3 \& \hw^4 \\
\hw^1 \& \hw^2 + A_2 (\hw^3 -C_3 \hthetabar) \& (\hw^4 -C_4\hthetabar) \\
B_1 \htheta \& \hthetabar + C_1 \hw^3 \& \hw^4 \\
B_2 \htheta \& \hthetabar + C_2 \hw^3 \& \hw^4 \\
B_3 \htheta \& \hthetabar + C_3 \hw^1 \& \hw^2 \\
B_4 \htheta \& \hthetabar + C_4 \hw^1 \& \hw^2
\end{bmatrix}.
\end{equation}
The connection form is not unique, but $A_i$, $B_i$, $C_i$ are well-defined torsion functions on $\P$ (see \cite{C02} for more details).
The relationship between the $G$-structure and the Pfaffian systems involved in the \BT/ is that,
if  $(\utheta,\uthetabar, \uw^1, \uw^2, \uw^3, \uw^4)$ is a local section of $\P$, then on its domain
$\pi^*\I=\igenl\utheta, \uw^1 \& \uw^2, \uw^3 \& \uw^4\igenr$
and $\pibar^*\Ibar = \igenl\uthetabar, \uw^1 \& \uw^2, \uw^3 \& \uw^4\igenr$.
Because $\utheta$ and $\uthetabar$ must each have Pfaff rank 5, $A_1$ and $A_2$ must be nonzero at
every point; furthermore, the condition of normality implies that $A_1 A_2 -1$ is also nonzero everywhere.

The $G$-structure endows $T^*\B$ with a well-defined splitting
\begin{equation}\label{Gsplit}
T^* \B =L \oplus \Lbar \oplus W_1 \oplus W_2
\end{equation}
such that, given any local section of $\P$, $L$ is spanned by $\utheta$, $\Lbar$ by
$\uthetabar$, $W_1$ by $\{\uw^1, \uw^2\}$, and $W_2$ by $\{\uw^3, \uw^4\}$.
In \cite{CI09} it is shown that the torsion functions are components of well-defined tensors on $\B$ which
are maps between bundles associated to terms in this splitting.  One way to see this is to study how these functions
vary along the fibers.  For example, if $g\in G$ is the group element given by \eqref{groupie}, then
$$R_g^* A_1 = a^{-1} b A_1,\qquad R_g^* A_2 =b^{-1} a A_2. $$
Notice that this implies that the product $A_1 A_2$ is a well-defined function on $\B$.
It also follows that $A_1$ and $A_2$ are components of well-defined tensors in $L \otimes \Lambda^2 W_2^*$
and $\Lbar \otimes \Lambda^2 W_1^*$, respectively.   Similarly,
the vectors $[C_1, C_2]$ and $[C_3, C_4]$ are components
of well-defined tensors
$\tau_1 \in \Gamma(W_1^* \otimes \Lambda^2 W_2)$
and $\tau_2 \in \Gamma(W_2^* \otimes \Lambda^2 W_1)$.
In fact, these tensors are
just the exterior derivative followed by an appropriate quotient map; for example
$\tau_1$ is simply the exterior derivative applied to sections of $W_1$, modulo
1-forms in $L, \Lbar$ and $W_1$.


For the sine-Gordon example above, one can take the following local section of $\P$
(recall that $\B \subset \R^5 \times \R^5$ is defined by $\xbar = x, \ybar = y$, and equations \eqref{SGEback-defineB}):
$$
\begin{aligned}
\theta & = du - p\, dx - q\, dy 
= du - \left(\pbar + 2 \lambda \sin\left(\frac{u+\ubar}{2}\right)\right)\, dx - q\, dy, \\
\thetabar & = d\ubar - \pbar\, dx - \qbar\, dy 
= d\ubar - \pbar\, dx - \left( -q + \frac{2}{\lambda}\sin\left(\frac{u-\ubar}{2}\right) \right)\, dy, \\
\omega^1 & = dx, \\
\omega^2 & = d\pbar - (\sin \ubar)\, dy + \lambda\cos\left(\frac{u+\ubar}{2}\right) \thetabar, \\
\omega^3 & = dy, \\
\omega^4 & = dq - (\sin u)\, dx - \frac{1}{\lambda} \cos\left(\frac{u-\ubar}{2}\right) \theta.
\end{aligned}
$$
The specific multiples of $\theta, \thetabar$ appearing in $\omega^2, \omega^4$ are uniquely determined by the conditions
\begin{equation}\label{some-d-conds}
\begin{aligned}
 d\theta & \equiv 0 \mod{ \theta, \omega^1 \& \omega^2, \omega^3 \& \omega^4 }, \\
 d\thetabar & \equiv 0 \mod{ \thetabar, \omega^1 \& \omega^2, \omega^3 \& \omega^4 },
\end{aligned}
\end{equation}
which are necessary to satisfy the structure equations \eqref{streq}.  The torsion functions associated to this local section are:
\begin{gather*}
A_1 = 1, \qquad A_2 = -1, \qquad 
B_1 = B_3 = C_1 = C_3 = 0, \\
B_2 = -\frac{\lambda}{2} \sin \left( \frac{u + \ubar}{2} \right), \qquad
B_4 = \frac{1}{2\lambda} \sin \left( \frac{u - \ubar}{2} \right), \\
C_2 = -\lambda \cos \left( \frac{u + \ubar}{2} \right), \qquad
C_4 = -\frac{1}{\lambda} \cos \left( \frac{u -\ubar}{2} \right).
\end{gather*}

\section{Wavelike, quasilinear, autonomous \BT/s}\label{results-sec}

In this section, we show how properties (\ref{startprop})-(\ref{endprop}) of the \BT/ \eqref{sgbt} may be characterized geometrically,
in terms of the invariants $A_i, B_i, C_i$ associated to the $G$-structure $\P$ determined by the exterior differential system $(\B, \J)$.

\subsection{Wavelike \BT/s}
A hyperbolic Monge-Amp\`ere PDE is called {\em wavelike} if it may be expressed in local coordinates as
\begin{equation}\label{waveform} u_{xy} = f(x,y,u,u_x,u_y);
\end{equation}
this is equivalent to the condition that the characteristics are tangent to the coordinate directions at each point.
In terms of these local coordinates, a \BT/ between two wavelike Monge-Amp\`ere PDEs is called {\em wavelike}
if it preserves the characteristic directions; this is equivalent to the condition that it preserves the characteristic independent variables $x,y$ up to a transformation of the form
\[ \xbar = \phi(x), \qquad \ybar = \psi(y). \]
By making an analogous change of independent variables for one PDE or the other, we may assume without loss of generality that a wavelike \BT/ satisfies the condition $\xbar = x, \ybar = y$.

A wavelike \BT/ between two wavelike PDEs
\[ u_{xy} = f(x,y,u,u_x,u_y), \qquad v_{xy} = g(x,y,v, v_x, v_y) \]
is generally described by equations of the form
\[ u_x = F(x,y,u,v,u_y, v_x), \qquad v_y = G(x,y,u,v,u_y,v_x). \]
(This will be made more precise in the proof of Proposition \ref{waveprop}.)
In this case the following coframing of $\B$ is a local section of $\P$:
\begin{align}
\theta & = du - F\, dx - q\, dy \notag \\
\thetabar & = d\ubar - \pbar\, dx - G\, dy \notag \\
\omega^1 & = dx, \label{wavelike-BT-coframing} \\
\omega^2 & = d\pbar - g\, dy + r_1 \thetabar, \notag\\
\omega^3 & = dy, \notag \\
\omega^4 & = dq - f\, dx - r_2 \theta, \notag
\end{align}
where $r_1, r_2$ are functions on $\B$ uniquely determined by the conditions \eqref{some-d-conds}.

Note that the Pfaffian systems $W_1 = \{\omega^1, \omega^2\}$ and $W_2 = \{\omega^3, \omega^4\}$
each contain a rank one integrable subsystem: $\{dx\} \subset W_1$ and $\{dy\} \subset W_2$.
It turns out that, conversely, this condition characterizes wavelike \BT/s; this is similar to the well-known result (see, e.g., \cite{Tunitskii}) that a hyperbolic Monge-Amp\`ere PDE is wavelike if and only each of its characteristic systems contains a rank 1 integrable subsystem.  In order to demonstrate this,
we give the following geometric definition for a wavelike \BT/, and in Proposition \ref{waveprop}
we will show that any \BT/ that satisfies this geometric condition is wavelike in the sense above.

\begin{definition}\label{wavelike-BT-def}A \BT/ between two hyperbolic Monge-Amp\`ere equations will be called {\em wavelike} if each of the Pfaffian systems $W_1, W_2$ is not integrable but contains a rank one integrable subsystem.
\end{definition}

\begin{remark}
In \cite{C02} it is shown that $W_1$ (resp. $W_2$) is integrable if and only if the tensor $\tau_1$ (resp., $\tau_2$) is identically zero.
\BT/s where one (or both) of $\tau_1, \tau_2$ vanishes are highly degenerate, and are classified in \cite{C02}.
Here, we will assume that both are nonzero at each point, which is equivalent to the vectors $[C_1, C_2]$, $[C_3, C_4]$ being both nonzero.
\end{remark}

\begin{prop}\label{waveprop} Let $\P \searrow \B$ define wavelike normal \BT/ which is wavelike
 in the sense of Defn. \ref{wavelike-BT-def}, and for which both tensors $\tau_1, \tau_2$ are nonvanishing on $\B$.
 Then near any point on $\B$
there exist local coordinates $x,y,u,v,p,q$ and functions $f,g,F,G$
such that the following is a section of $\P$:
\begin{equation}\label{wavecof}
\begin{aligned}
\utheta &= du - F \,dx - q\,dy, \qquad & \uw^1 &= dx,\qquad &\uw^2 &= dp -g\,dy - (F_v/F_p)\uthetabar,\\
\uthetabar &= dv - p \,dx - G\, dy, & \uw^3 &=dy,\qquad & \uw^4 &= dq -f \,dx - (G_u/G_q) \utheta ,
\end{aligned}
\end{equation}
and $f,g,F,G$ satisfy the following partial differential equations:
\begin{gather}
F_q = G_p = 0, \label{pde1} \\
f - g F_p = F_y +  q F_u+G F_v \label{solvefg1},\\
g - f G_q = G_x + F G_u + p\, G_v \label{solvefg2},\\
0 = f_v F_p - f_p F_v = g_u G_q - g_q G_u. \label{dropFG}
\end{gather}
Moreover, the quantities $F_p, G_q,$ and $F_p G_q - 1$ are nonzero at every point of $\B$;
in particular, equations \eqref{solvefg1}, \eqref{solvefg2} can be solved for $f$ and $g$.
In these coordinates, $\B$ represents a \BT/ between the wavelike partial differential equations
\begin{equation}
 u_{xy} = f(x,y,u, u_x, u_y), \qquad v_{xy} = g(x,y,v,v_x,v_y), \label{nonlinear-wave-eqs}
\end{equation}
with the transformation given by the equations
\begin{equation}
u_x = F(x,y,u,v,v_x), \qquad v_y = G(x,y,u,v,u_y). \label{wavelike-BT}
\end{equation}

\end{prop}
\begin{proof}
Take a local section of $\P$ such that $\uw^1$ and
$\uw^3$ span the integrable subsystems of $W_1$ and $W_2$, respectively.  Using the $G$-action,
we can modify the section (specifically, by multiplying $\uw^1,\uw^3$ by suitable scaling functions)
so that $\uw^1=dx$ and $\uw^3=dy$ for some locally defined functions $x, y$ on $\B$.
The structure equations \eqref{streq} imply that $\{\utheta, dx, dy\}$ is a Frobenius system; therefore, there are a locally defined functions $u, p_1, q_1$ on $\B$ such that
\[\utheta = \mu( du -p_1 \,dx - q_1 \,dy)\]
for some nonvanishing multiple $\mu$. Moreover, because $\utheta$ has Pfaff rank 5, the functions $u,x,y,p_1, q_1$ must have linearly independent differentials.  Similarly,
there must exist locally defined functions $v,p_2, q_2$ such that
\[ \uthetabar = \other{\mu}( dv - p_2 \,dx - q_2 \,dy)\]
for some nonvanishing multiple $\other{\mu}$.  Using the $G$-action, we can modify the
section by scaling so that $\mu=\other{\mu}=1$.

The structure equations imply that $d\utheta \equiv A_1 dx \& \uw^2+ dy \& \uw^4$ modulo $\utheta$.
Substituting $\theta = du - p_1 \,dx - q_1 \,dy$ into this equation shows  that we must have
\begin{align}
A_1 \uw^2 &= dp_1 - s_1 \,dx - f \,dy - r_3 \utheta, \label{A1w2}\\
    \uw^4 &= dq_1 - f \,dx - s_2 \,dy - r_2 \utheta	\label{uw4} \\
\intertext{for some functions $f,r_2,r_3,s_1,s_2$.  Similarly, substituting $\thetabar = dv - p_2 \,dx - q_2 \,dy$ into the equation
 $d\uthetabar \equiv dx \& \uw^2+ A_2 dy \& \uw^4$ mod $\thetabar$ shows that}
    \uw^2 &= dp_2 - s_3 \,dx - g\,dy - r_1 \uthetabar, \label{uw2} \\
A_2 \uw^4 &= dq_2 -g \,dx - s_4 \,dy - r_4 \uthetabar  \label{A2w4}
\end{align}
for some functions $g,r_1,r_4,s_3,s_4$.
We may use the remaining freedom in the $G$-action to modify the section, by adding multiples of $dx$ and $dy$ to
$\uw^2$ and $\uw^4$, respectively, to arrange that $s_2=0$ and $s_3=0$.

Comparing \eqref{A1w2} and \eqref{uw2} shows that the functions $p_1, p_2, x, y, u,v$ are functionally
dependent, and comparing \eqref{A2w4} and \eqref{uw4} shows that the functions $q_1, q_2, x, y, u, v$ are
functionally dependent; on the other hand, linear independence of
the forms $\utheta, \uthetabar, dx, dy, \uw^2$ and $\uw^4$ shows that
$x,y,u,v,p_2$ and $q_1$ are a local coordinate system on $\B^6$.
Thus, locally there exist functions $F$ and $G$ such that
$$p_1 = F(x,y,u,v,p_2)\qquad q_2 = G(x,y,u,v,q_1).$$
Substituting these equations, together with equations \eqref{uw4}, \eqref{uw2}, and the expressions for $\utheta$ and $\uthetabar$ into equations \eqref{A1w2} and \eqref{A2w4} yields
\begin{equation}\label{eqs-to-compare}
\begin{aligned}
A_1(dp_2 - g\,dy - r_1 (dv - p_2\,dx - G\,dy)) &= dF - s_1\,dx - f\,dy - r_3(du -F\,dx - q_1\,dy)\\
A_2(dq_1 - f\,dx - r_2(du-F\,dx - q_1\,dy)) &= dG - g\,dx - s_4\, dy - r_4(dv -p_2 \,dx - G\,dy)).
\end{aligned}
\end{equation}
For convenience, set $p = p_2$ and $q = q_1$.  Expanding $dF, dG$ and equating the coefficients of the differentials of the coordinates in equations \eqref{eqs-to-compare} yields
$$
\begin{aligned}
A_1 &= F_p, 	\quad & r_1 &= -F_v/F_p, &\quad r_3 &= F_u, \quad & f - F_p g &= F_y + F_u q + F_v G,\\
A_2 &= G_q,	& r_2 &= -G_u/G_q, & r_4 &=G_v, & g - G_q f, &= G_x + G_u F + G_v p.
\end{aligned}
$$
This establishes equations \eqref{solvefg1} and \eqref{solvefg2} in the statement of the proposition,
as well as the form of the coefficients of $dx$ in $\uw^2$ and $dy$ in $\uw^4$ in \eqref{wavecof}.

The structure equations \eqref{streq} imply that
$$d\hw^4 \& \hw^3 \& \hw^4 \& \htheta = \dfrac{C_4}{2 A_1} d\htheta \& d\htheta \& \htheta.$$
Pulling this condition back to our particular section gives
$$-df \& dx \& dy \& dq \& du = \dfrac{C_4}{A_1} dF \& dx \& dq \& dy \& du.$$
The local invariants for this section satisfy $C_1 = C_3 = 0$; therefore, our assumption that $W_1, W_2$
are not completely integrable implies that the functions $C_2, C_4$ are nonzero.
Thus it follows from this equation that $df \& dF \equiv 0$ modulo $dx, dy, du$ and $dq$.
This establishes the first equation on the last line
\eqref{dropFG} in the proposition. Moreover, it implies that $f$ may be regarded as a function of the
variables $x,y,u, p_1,q_1$; hence $f$ is locally a well-defined function on $\M$.  Similarly, the second equation in \eqref{dropFG} may be derived from the condition
$$d\hw^2 \& \hw^1 \& \hw^2 \& \hthetabar = \dfrac{C_2}{2 A_2} d\hthetabar\& d\hthetabar \& \hthetabar,$$
and $g$ may be regarded as a function of the variables $x,y,v, p_2,q_2$; hence $g$ is locally a well-defined function on $\Mbar$.

Since we have $A_1 = F_p, A_2 = G_q$, the condition that $\B$ is a normal \BT/ implies that $F_p, G_q \neq 0$ and $F_p G_q - 1 \neq 0$ at every point of $\B$.  Any integral
surface of $\J$ on which $dx \& dy \neq 0$ is given by specifying $u,v,p,q$ as functions of $x$ and $y$
that satisfy the conditions $p=v_x$, $q=u_y$, and
\begin{equation}\label{waveBT}
u_x = F(x,y,u,v,p), \qquad v_y = G(x,y,u,v,q).
\end{equation}
By taking total derivatives of these equations with respect to $y$ and $x$ respectively, it follows that
$u$ and $v$ satisfy the wavelike PDEs \eqref{nonlinear-wave-eqs},
where $f,g$ are the functions determined by equations \eqref{solvefg1},\eqref{solvefg2}.

\end{proof}

\subsection{Quasilinear wavelike \BT/s}
The \BT/ \eqref{sgbt} for the sine-Gordon equation is defined by {\em quasilinear} PDEs---i.e.,
PDEs that are linear in the partial derivatives $u_x, u_y, v_x, v_y$.  For the general case of a
wavelike \BT/ between two wavelike PDEs, this condition may be expressed in terms of the local
coordinates given by Proposition \ref{waveprop} as the condition that the functions $F, G$ are
linear in the variables $p,q$, respectively.  In this case, the invariants associated to the
section \eqref{wavecof} of $\P$ have the property that the functions $A_1 = F_p$ and $A_2 = G_q$
are functions of the variables $x,y,u,v$ alone.  It follows that the product $A_1 A_2$,
which is well-defined on $\B$ {\em independent} of the choice of section of $\P$, is also a function of the variables $x,y,u,v$ alone.

This condition can be expressed geometrically as follows: any wavelike \BT/ $\B$ is endowed with a well-defined rank 4 Pfaffian system $K$ that is the direct
sum of $L$, $\Lbar$ and the integrable subsystems of $W_1$ and $W_2$; moreover, this system is Frobenius.  (In terms of the local coordinates given
by Proposition \ref{waveprop}, $K$ is spanned by $\{dx, dy, du, dv\}$.)

\begin{definition}\label{quasilinear-BT-def}
A wavelike \BT/ between two wavelike hyperbolic Monge-Amp\`ere equations will be called {\em quasilinear}
if the function $A_1 A_2$ is constant along the integral submanifolds of $K$ (or equivalently, if $d(A_1 A_2) \in \Gamma(K)$).
\end{definition}

In terms of the local normal form \eqref{waveBT} given by Proposition \ref{waveprop}, this condition is easily seen to be equivalent to the condition that the functions $F$ and $G$
are linear in the variables $p$ and $q$, respectively.
Thus, in the quasilinear case we may set
\begin{equation}\label{FGlinear}
F = F_0 + F_1 p, \qquad G=G_0 + G_1 q
\end{equation}
where $F_i$ and $G_i$ are functions of the variables $x,y,u,v$, and the quantities $F_1, G_1$ and $1-F_1 G_1$ are nonzero
at each point.  Then the equations \eqref{waveBT} defining the \BT/ become:
\begin{equation}\label{quasiBT}
u_x = F_1 v_x + F_0, \qquad v_y = G_1 u_y + G_0.
\end{equation}

The following proposition shows that in the quasilinear case, the local coordinates of
Proposition \ref{waveprop} may be refined in such a way that the functions $f, g$ become linear
with respect to the variables $p,q$.  Consequently, the PDEs \eqref{nonlinear-wave-eqs} underlying
a quasilinear \BT/ \eqref{FGlinear} may both be assumed to be quasilinear as well.

\begin{prop}\label{quasiFFGG}
Let $\P \searrow \B$ define a quasilinear wavelike normal \BT/.  Near any point on $\B$
there exist local coordinates $x,y,u,v,p,q$ and functions $f,g,F,G$ satisfying the conditions of Proposition \ref{waveprop}, together with the
additional conditions that $F$ is linear with respect to $p$, $G$ is linear with respect to $q$,
and the functions $f, g$ are linear with respect to the variables $p,q$.  (In particular, $f$ and $g$ contain no terms involving the product $pq$.)
\end{prop}

\begin{proof}  Let $x,y,u,v,p,q$ be the local coordinates provided by Proposition \ref{waveprop}.
As noted above, the assumption that $\B$ is quasilinear implies that $F,G$ have the form \eqref{FGlinear}.
Substituting these expressions into equations \eqref{solvefg1}, \eqref{solvefg2} yields the following system of equations for $f,g$:

\[ \begin{bmatrix} 1 & -F_1 \\ -G_1 &1\end{bmatrix} \begin{bmatrix} f \\ g \end{bmatrix}= \begin{bmatrix}
F_{0,y}+G_0 F_{0,v}  & F_{1,y}+G_0 F_{1,v}  & F_{0,u} + G_1 F_{0,v} & F_{1,u} + G_1 F_{1,v} \\[5pt]
G_{0,x}+F_0 G_{0,u} & G_{0,v} + F_1 G_{0,u} & G_{1,x}+F_0 G_{1,u} & G_{1,v} + F_1 G_{1,u}
\end{bmatrix}
\begin{bmatrix} 1 \\ p \\ q \\ pq \end{bmatrix}. \]
Therefore, $f$ and $g$ are given by:
\begin{equation}\label{expandpq}
\begin{bmatrix} f \\ g \end{bmatrix}=
\dfrac1\Delta\begin{bmatrix} 1 & F_1 \\ G_1 &1\end{bmatrix}
\begin{bmatrix}
F_{0,y}+G_0 F_{0,v}  & F_{1,y}+G_0 F_{1,v}  & F_{0,u} + G_1 F_{0,v} & F_{1,u} + G_1 F_{1,v} \\[5pt]
G_{0,x}+F_0 G_{0,u} & G_{0,v} + F_1 G_{0,u} & G_{1,x}+F_0 G_{1,u} & G_{1,v} + F_1 G_{1,u}
\end{bmatrix}
\begin{bmatrix} 1 \\ p \\ q \\ pq \end{bmatrix},
\end{equation}
where $\Delta = 1 - F_1 G_1$.

As is evident from \eqref{expandpq} and the linearity relations \eqref{FGlinear}, such transformations link solutions of hyperbolic equations of the form
\begin{equation}\label{unonlin}
u_{xy} = A u_x u_y +B u_x +  C u_y + D
\end{equation}
where $A, B, C, D$ are functions of $x,y,u$.
Moreover, the form of such equations, along with the form \eqref{quasiBT} for the
transformation, is invariant under point transformations defined by
$U=\varphi(u,x,y)$.
In order to complete the proof, it remains to show that we can use such changes of variable to eliminate the
first-order nonlinearity in the right-hand side of \eqref{unonlin}.

Under a change of variable $U=\varphi(x,y,u)$, the PDE \eqref{unonlin} is transformed to the PDE
$$U_{xy} = \widetilde{A} U_x U_y + \widetilde{B} U_x + \widetilde{C} U_y + \widetilde{D}$$
for the function $U(x,y)$,
where, in particular,
$$\widetilde{A} = \dfrac{\varphi_{uu} + A \varphi_u}{\varphi_u^2}.$$
By choosing $\varphi(x,y,u)$ so that it satisfies the first-order PDE
$$\varphi_u = \int e^{-A(x,y,u)} \,du,$$
we can arrange that $\varphi_{uu} + A\varphi_u=0$; thus, the PDE satisfied by $U$
has no nonlinear first-order term.  Similarly, we can simultaneously make a change of variable $V=\psi(x,y,v)$,
to arrange that the PDE satisfied by $V$ has no nonlinear first-order term. When we do so, the quasilinear wavelike form \eqref{quasiBT}
of the \BT/ is unchanged.  (However, the coefficients in \eqref{quasiBT} are altered; for example,
$F_1$ is replaced by $\frac{\varphi_u F_1}{\psi_v}$ and $G_1$ is replaced by $\frac{\psi_v G_1}{\varphi_u}$.)
Now re-labeling $U$ as $u$ and $V$ as $v$ gives the desired local coordinates.
\end{proof}

\subsection{Autonomous wavelike \BT/s}
A PDE is called {\em autonomous} if it contains no explicit dependence on the independent variables $x,y$.
Thus, the wavelike \BT/ \eqref{waveBT} is autonomous if $F$ and $G$ are functions of the variables $u,v,p,q$ alone,
with no dependence on the variables $x,y$.  In this case, we can see from equations \eqref{solvefg1}, \eqref{solvefg2}
that the functions $f,g$ would also be independent of the variables $x,y$, and hence the PDEs underlying the \BT/ would have the form
\[ u_{xy} = f(u, u_x, u_y), \qquad v_{xy} = g(v, v_x, v_y). \]

Geometrically, the condition that the system is autonomous is represented by the presence of a
2-dimensional Abelian symmetry group of the exterior differential system $\J$ on $\B$---namely,
the group of simultaneous translations in the independent variables $x,y$.  Specifically,
if $\J$ represents an autonomous \BT/, then for any real numbers $a,b$, the diffeomorphism $\phi:\B \to \B$ defined by
\[ \phi(x,y,u,v,p,q) = (x+a, y+b, u,v,p,q) \]
has the property that $\phi^*\J = \J$.

It is generally more convenient to work with the {\em infinitesimal symmetries} of $\J$.  These are the vector fields on $\B$ that generate the
symmetry group; in the case of the translation symmetry group above, the infinitesimal symmetries are generated by the two commuting vector fields
\[ X = \frac{\partial}{\partial x}, \qquad Y = \frac{\partial}{\partial y}. \]

In general, we have the following definition:

\begin{definition}\label{symmetry-def}
A vector field $X$ on a manifold $\B$ is an {\em infinitesimal symmetry} of the exterior differential system $\J$ on $\B$
if for every differential form $\Phi \in \J$, the Lie derivative $\Lie_X \Phi$ is contained in $\J$.
\end{definition}

Since the $G$-structure $\P \searrow \B$ is canonically associated to the exterior differential system
$\J$ on $\B$, any 
symmetry $\phi$
of $(\B, \J)$ must also be 
a symmetry
of $\P$, and vice-versa.
In other words, for any local section $\unter\vomega$ of $\P$,
$\phi^* \vomega$
must also be a local section of $\P$.
In particular,
if $X$ is an infinitesimal symmetry then
the Lie derivative $\Lie_X$ must preserve the splitting \eqref{Gsplit} of the cotangent bundle $T^*\B$.
Furthermore, if the \BT/ is wavelike, then $\Lie_X$ must also preserve
the 1-dimensional integrable subsystems of $W_1$ and $W_2$.

The following proposition shows that in the wavelike case, the presence of a pair of
commuting infinitesimal symmetries characterizes the autonomous examples; more precisely,
if a wavelike \BT/ has two linearly independent, commuting infinitesimal symmetries
(subject to a transversality condition which will be made precise below),
then there exist local coordinates with respect to which the PDEs defining the \BT/ are autonomous.

\begin{prop}\label{wavelikesymm}
 Let $\P\searrow \B$ define a normal wavelike \BT/.  Let $X$ and $Y$ be pointwise linearly independent,
commuting vector fields on $\B$ which are infinitesimal symmetries of $\P$ and are
{\em transverse} to the integrable subsystems of $W_1$ and $W_2$; i.e., if
$\uw^1$ and $\uw^3$ are local sections of these integrable subsystems, then
\begin{equation}\label{transverse}
\det\begin{bmatrix} \w^1(X) & \w^1(Y) \\ \w^3(X) & \w^3(Y) \end{bmatrix} \ne 0
\end{equation}
at each point of $\B$.  Then near any point in $\B$ there exist local coordinates $x,y,u,v,p,q$
and functions $F,G,f,g$ satisfying the conditions of Proposition \ref{waveprop}, with the additional property that $F,G,f,g$ are independent of $x$ and $y$.
\end{prop}

\begin{remark}
This result is not as obvious as it may seem.  Certainly, the existence of commuting symmetry vector fields easily implies the existence of local coordinates $x,y,u,v,p,q$ in which the \BT/ may be expressed as a system of PDEs that have no explicit dependence on the variables $x,y$.  But showing that these coordinates may be chosen consistently with the conditions of Proposition \ref{waveprop} requires a bit more care than might be expected at first glance.
\end{remark}

\begin{proof}
Because the vector fields $X$ and $Y$ commute, they are tangent to a local foliation of $\B$ with two-dimensional leaves.
Thus, any point in $\B$ has a neighborhood on which there exist functions $u,v,p,q$ whose differentials are linearly independent and annihilate $X$ and $Y$, i.e.,
\[ \{du, dv, dp, dq\}^\perp = \{X, Y\}. \]
Moreover, the span of $\{du, dv, dp, dq\} \subset T^*\B$ is uniquely determined by this condition.

Let $x,y$ be locally defined functions, possibly defined on a smaller neighborhood of the given point,
such that $dx$ and $dy$ span the integrable
subsystems of $W_1$ and $W_2$, respectively.  Then by our hypothesis \eqref{transverse} the functions
$x,y,u,v,p,q$ form a local coordinate system, and
\[ \text{span} \left\{X, Y \right\} = \text{span} \left\{\dfrac{\di}{\di x}, \dfrac{\di}{\di y} \right\} \]
at each point.

First, we will show that we can replace $X,Y$ by {\em constant coefficient} linear combinations
\begin{equation}
 \tilde{X} = aX + bY, \qquad \tilde{Y} = cX + dY \label{XY-new-basis}
\end{equation}
(so that $\tilde{X}, \tilde{Y}$ commute and are infinitesimal symmetries of $\B$ as well) and $x,y$ by functions $\tilde{x}(x), \tilde{y}(y)$ to arrange that
\[ \tilde{X} = \dfrac{\di}{\di \tilde{x}}, \qquad \tilde{Y} = \dfrac{\di}{\di \tilde{y}}. \]
To this end, note that if $Z$ is any symmetry vector field of $(\B, \J)$, then
\[ \Lie_Z dx =d(Z\intprod dx)\equiv 0 \mod dx, \]
and similarly, $d(Z\intprod dy) \equiv 0$ mod $dy$.  In other words, the function $Z \intprod dx$ is a function of $x$ alone, and $Z \intprod dy$ is a function of $y$ alone.
Thus, if we let $M$ be the matrix
$$M = \begin{bmatrix} X\intprod dx & Y\intprod dx \\ X\intprod dy & Y\intprod dy \end{bmatrix},$$
then the top row entries $M_{11}, M_{12}$ are functions of $x$ alone, while the bottom row entries $M_{21}, M_{22}$ are functions of $y$ alone.
Moreover, by replacing $X,Y$ by constant-coefficient linear combinations $\tilde{X},\tilde{Y}$ as in \eqref{XY-new-basis}, we can
assume that $M_{12}$ and $M_{21}$ vanish at the given point of $\B$.  It follows that $M_{12}$ also vanishes on the hypersurface through the given point where
$x$ is constant, while $M_{21}$ vanishes on the hypersurface through the given point where $y$ is constant.
But we also have
\begin{align*}
\tilde{X}(M_{12}) & = \tilde{X}(\tilde{Y} \intprod dx) = \tilde{X}(\tilde{Y}(x)) & &\text{(by definition)} \\
& = \tilde{Y}(\tilde{X}(x)) &&\text{(because $\tilde{X}, \tilde{Y}$ commute)} \\
& = \tilde{Y}(M_{11})  = M_{11}'(x) \tilde{Y}(x)&& \text{(by the chain rule)} \\
& = M_{11}'(x) M_{12}.
\end{align*}
So, by the local uniqueness theorem for ODEs, $M_{12}$ vanishes identically in a neighborhood of the given point,
and a similar argument shows that the same is true for $M_{21}$.  By solving the ODE
\[ \frac{d\tilde{x}}{dx} = \frac{1}{M_{11}(x)} \]
we obtain a function $\tilde{x}$ of $x$ satisfying $\tilde{X}(\tilde{x})=1$.
Similarly, we obtain a function $\tilde{y}$ of $y$ such that $\tilde{Y}(\tilde{y})=1$.  Dropping
the tildes, we obtain a local coordinate system $x,y,u,v,p,q$ such that
$$X = \dfrac{\di}{\di x}, \qquad Y=\dfrac{\di}{\di y},$$
as claimed.

Next, we will show that we can choose a nonvanishing local section $\theta_1$ of $L$ of the form
\begin{equation*}
\theta_1 = A \, du + B \, dv + C\,dp + D\,dq + R \,dx + S\,dy,
\end{equation*}
where $A,B,C,D,R,S$ are locally defined functions on $\B$ that are independent of the variables $x,y$.
Start by choosing any nonvanishing local section  $\theta_0$ of $L$.
(Here and in what follows we will shrink our local coordinate neighborhood around the given point as necessary.)
Because $\{\theta_0,dx,dy\}$ is a Frobenius system, there exist locally defined functions $U,P,Q$ such that, up to a scalar multiple,
\begin{align}
 \theta_0 & = dU - P\,dx - Q\,dy  \label{firstheta}\\
 & = U_{u}\, du + U_v\, dv + U_{p}\, dp + U_{q}\, dq + (U_{x} - P)\, dx + (U_{y} - Q)\, dy . \notag
\end{align}
Because $X$ and $Y$ are symmetries of the system, we must have
\[ \Lie_X \theta_0 \equiv \Lie_Y \theta_0 \equiv 0 \mod{\theta_0}. \]
Direct computation using equation \eqref{firstheta} shows that
\begin{align*}
\Lie_X \theta_0 & = U_{xu}\, du + U_{xv}\, dv + U_{xp}\, dp + U_{xq}\, dq + (U_{xx} - P_x)\, dx + (U_{xy} - Q_x)\, dy, \\
\Lie_Y \theta_0 & = U_{yu}\, du + U_{yv}\, dv + U_{yp}\, dp + U_{yq}\, dq + (U_{xy} - P_y)\, dx + (U_{yy} - Q_y)\, dy.
\end{align*}
Each of these must be a scalar multiple of $\theta_0$. Expanding $0= \theta_0 \wedge \Lie_X \theta_0$ in our coordinates
shows that the $x$-derivative of the ratio any two of the functions
\[ U_u, U_v, U_p, U_q, (U_x - P), (U_y - Q) \]
must be zero; similarly, these ratios must also be independent of $y$.
Then if, say, $U_u \neq 0$, we can write
\[ \theta_0 = U_u \left( du + \frac{U_v}{U_u} dv + \frac{U_p}{U_u} dp + \frac{U_q}{U_u} dq + \frac{(U_x-P)}{U_u} dx + \frac{(U_y-Q)}{U_u} dy \right) .\]
It follows that we may take $\theta_1 = e^{-\lambda} \theta_0$ where $e^\lambda = U_u$, so that
\begin{equation}\label{thetatilde}
\theta_1 = A \, du + B\, dv + C\,dp + D\,dq + R \,dx + S\,dy
\end{equation}
for functions some $A,B,C,D,R,S$ that are independent of the variables $x,y$, as claimed.
A similar argument shows that there exists a nonvanishing section $\thetabar_1$ of $\Lbar$ of the form
\[ \thetabar_1 = \other{A}\, du + \other{B} \, dv + \other{C}\,dp + \other{D}\,dq + \other{R} \,dx + \other{S}\,dy, \]
where the functions $\other{A}, \other{B}, \other{C}, \other{D}, \other{R}, \other{S}$ are independent of the variables $x,y$.

Finally, we will show that we can modify our local coordinates and
rescale the sections $\theta_1, \thetabar_1$ to arrive at sections $\theta, \thetabar$ of $L, \Lbar$, respectively, of the form
\begin{align*}
\theta &= du - F(u,v,p) \,dx - q \,dy,\\
            \thetabar&= dv - p \,dx - G(u,v,q)\, dy
\end{align*}
for some functions $F, G$ that are independent of the variables $x,y$.  It will follow that the \BT/ is given in terms of these local coordinates the the equations
\[ u_x = F(u,v,v_x), \qquad      v_y =G(u,v,u_y), \]
which will complete the proof of the Proposition.
To simplify the exposition, we introduce the following notations: let $I_1 \subset T^* \B$ and $I_2 \subset T^*\B$
be the complementary local sub-bundles spanned by $\{dx,dy\}$ and $\{du,dv,dp,dq\}$, respectively.
Let $\pi_1, \pi_2$ denote projections onto these sub-bundles.  There is an induced splitting of the local 2-forms on $\B$, namely
\begin{equation}\label{2-form-splitting}
\Lambda^2 (T^* \B) = \Lambda^2 I_1 \, \oplus \,(I_1 \wedge I_2) \, \oplus \, \Lambda^2 I_2.
\end{equation}
We similarly let $\pi_{11}$ and $\pi_{22}$ denote projections onto the first and last summands in equation \eqref{2-form-splitting}, respectively.

By equation \eqref{firstheta}, we have
\[ \pi_{22}(d\theta_0) = \pi_{22}(-dP \& dx - dQ \& dy) = 0. \]
Substituting $\theta_0 = e^{\lambda}\theta_1$ into this equation yields (after cancelling a factor of $e^\lambda$)
\begin{equation}\label{pied}
\pi_{22} \left(d\lambda \wedge \theta_1 + d\theta_1\right) = 0.
\end{equation}
Let $\phi = \pi_1(\theta_1) =  A \, du + B\, dv + C\,dp + D\,dq$ and $\psi = \pi_2(\theta_1) = R\,dx + S\,dy.$
Substituting $\theta_1 = \phi + \psi$ into equation \eqref{pied} yields
\begin{align*}
0 &= \pi_{22}\left( d\lambda\wedge\phi + d\lambda \wedge \psi + d\phi + d\psi\right) \\
&= \pi_2(d\lambda) \wedge \phi + d\phi,
\end{align*}
where in the second line we have used the facts that $\pi_2(\psi)=0$
(and hence $\pi_{22}(d\lambda \& \psi) = 0$), $\pi_{22}(d\psi)=0$, and $\pi_{22}(d\phi) = d\phi$
(which follows from the fact that $A,B,C,D$ are independent of $x,y$).
But this implies that the 1-form $\phi$ is integrable, and by the Pfaff Theorem, locally there must exist functions $\wt{U},\mu$
 such that
$$\phi = e^\mu d\wt{U}.$$
Moreover, since $\phi$ has no dependence on the variables $x,y$, we can assume that $\wt{U}, \mu$ have no dependence on $x,y$ as well.
Thus, we can define a nonvanishing local section $\theta$ of $L$ by
$$\theta = e^{-\mu} \theta_1 = d\wt{U} - \wt{P}_1\,dx - \wt{Q_1}\,dy,$$
where $\wt{P}_1= e^{-\mu}R$ and $\wt{Q}_1 = e^{-\mu}S$ are functions of the variables $u,v,p,q$, with no dependence on the variables $x,y$.
By a similar argument, there exist functions $\wt{V}$, $\wt{P}_2$ and $\wt{Q}_2$ of the variables $u,v,p,q$ such that
$$\thetabar = d\wt{V} - \wt{P}_2 \,dx - \wt{Q}_2 \,dy.$$
is a nonvanishing local section of $\Lbar$.

Now we re-label the coordinates $\wt{U}, \wt{V}, \wt{P}_2, \wt{Q}_1$ as $u,v,p,q$ respectively.
Then we have
\begin{align*}\theta &= du - \wt{P}_1(u,v,p,q) \,dx - q \,dy,\\
            \thetabar&= dv - p \,dx - \wt{Q}_2(u,v,p,q)\, dy.
\end{align*}
Set $F=\wt{P}_1$ and $G=\wt{Q}_2$.
As was shown in the proof of Proposition \ref{waveprop}, $F$ is in fact independent of $q$, and
$G$ is independent of $p$.  Thus, the
\BT/ is given by
\begin{align*}u_x &= F(u,v,v_x),\\
            v_y &=G(u,v,u_y),
\end{align*}
where there is no explicit dependence on $x$ or $y$ in the right-hand sides, as claimed.
As noted previously, it follows that the right-hand sides $f,g$ of the \MA/ equations \eqref{nonlinear-wave-eqs}
are also independent of $x$ and $y$.
\end{proof}

In the quasilinear case, we have the following corollary:

\begin{cor}
Let $\P \searrow \B$ define a quasilinear wavelike \BT/, and assume that there exist vector fields $X,Y$
on $\B$ that satisfy the conditions of Proposition \ref{wavelikesymm}.  Then near any point
of $\B$ there are coordinates $x,y,u,v,p,q$ and functions $F_0,F_1,G_0,G_1$
satisfying the conclusions of Proposition \ref{quasiFFGG} and which are independent of $x$ and $y$.
\end{cor}


\section{Classifying \BT/s of sine-Gordon type}\label{classification-sec}


In order to classify wavelike normal \BT/s for which both tensors $\tau_1, \tau_2$ are nonvanishing,
 as in Proposition \ref{waveprop}, it would be necessary to find all solutions of the PDE system
 \eqref{pde1}-\eqref{dropFG} satisfying the appropriate nondegeneracy conditions.
After using equations \eqref{solvefg1}, \eqref{solvefg2} to eliminate $f, g$,
the remaining equations are an overdetermined system of four PDEs for the two functions $F, G$.
The Cartan theory of exterior differential systems~\cite{CFB2} provides a powerful technique for
determining the solution space for such overdetermined systems; unfortunately, for this system the required computations appear to be intractable in full generality.  However, if we restrict our attention to the autonomous, quasilinear case, then we can characterize the solution space to the corresponding PDE system \eqref{pde1}--\eqref{dropFG}.

Thus, we suppose that $\P \searrow \B$ is a quasilinear, wavelike, autonomous normal \BT/.  The results of \S \ref{results-sec} imply that there exist local coordinates $x,y,u,v,p,q$ on $\B$, functions $\alpha, \kappa, \mu$ of $u$ alone, and functions $\beta, \eta, \xi$ of $v$ alone
for which the functions $F,G,f,g$ may be written as
\begin{equation}\label{intro-FGfg}
\begin{gathered}
 F = F_0(u,v) + F_1(u,v) p, \qquad G = G_0(u,v) + G_1(u,v) q, \\
f = \mu(u) + \alpha(u) \bar{p} + \kappa(u) q = \mu(u) + \alpha(u) F + \kappa(u) q , \\
g = \xi(v) + \eta(v) p + \beta(v) \bar{q} = \xi(v) + \eta(v) p + \beta(v) G
\end{gathered}
\end{equation}
with $F_1 G_1 \neq 0$ or $1$.
Substituting the expressions \eqref{intro-FGfg} into the system \eqref{pde1}--\eqref{dropFG} and comparing coefficients with respect to $p$ and $q$ yields the following system of 8 PDEs:
\begin{subequations}\label{bigsys}
\begin{align}
F_{1,u} + G_1 F_{1,v} &= 0, &  G_{1,v} + F_1 G_{1,u} &=0, \label{topline} \\
G_0 F_{1,v} &= (\alpha-\eta)F_1, &  F_0 G_{1,u} &= (\beta-\kappa)G_1, \label{secondline} \\
F_{0,u} + G_1 F_{0,v} &= \kappa - F_1 G_1 \beta, &  G_{0,v} + F_1 G_{0,u} &=\eta - F_1 G_1 \alpha,
\label{thirdline} \\
G_0 F_{0,v} &= \mu+\alpha F_0 -(\xi + \beta G_0)F_1,
 &  F_0 G_{0,u} &=\xi + \beta G_0-(\mu+\alpha F_0)G_1.\label{fourthline}
\end{align}
\end{subequations}

\begin{remark*}
If one were trying to solve for a quasilinear \BT/ between two given equations $u_{xy} = \mu + \alpha u_x + \kappa u_y$
and $v_{xy} = \xi + \eta v_x + \beta v_y$, then \eqref{bigsys} constitutes an overdetermined system of 8 equations
for the 4 functions $F_0$, $F_1$, $G_0$, $G_1$.  One expects the solvability conditions for this system
(typically obtained by differentiating and equating mixed partials) to severely limit the space of solutions,
and for a generic pair of equations one expects it to be empty.  Even if we do not prescribe the underlying equations
related by the transformation, the system is still highly overdetermined, for we must augment \eqref{bigsys}
by the requirement that the $v$-partials of $\alpha, \kappa, \mu$ are zero
and the $u$-partials of $\beta, \eta, \xi$ are zero, yielding a system of 14 equations for 10 functions.
\end{remark*}

In order to analyze the solution space of the system \eqref{bigsys}, we must first consider separately the cases where the product $F_0 G_0$ is either zero or nonzero.  (If $F_0 G_0$ is not identically equal to zero, then we restrict our attention to the open set of $\B$ on which it is nonzero.)

\subsection*{Case 0: $F_0 G_0 = 0$.}  Without loss of generality, assume that $G_0 = 0$.  From the first equation in \eqref{secondline} and the fact that $F_1 \neq 0$, it follows that $\alpha(u) = \eta(v)$, and hence both $\alpha$ and $\eta$ must be equal to some constant $c$.  Then from the second equation in \eqref{thirdline} and the fact that $F_1 G_1 \neq 1$, it follows that $\alpha = \eta = 0$.
Finally, substituting these conditions into equations \eqref{fourthline} (and using the fact that $F_1 G_1 \neq 1$) yields $\mu = \xi = 0$.

Now, from equations \eqref{intro-FGfg}, we see that $\B$ is a \BT/ between two PDEs of the form
\[ u_{xy} = \kappa(u) u_y, \qquad v_{xy} = \beta(v) v_y. \]
These PDEs are both Monge-integrable.  \BT/s between such PDEs are not particularly interesting, because the underlying PDEs can already be solved using only ODE techniques.

\bigskip
Thus, for the rest of this section we will assume that $F_0 G_0 \ne 0$.  Before proceeding with our analysis, we prove the following theorem regarding the size of the solution space:

\begin{theorem}\label{finite-dimension-thm}
The space of solutions to the system \eqref{bigsys} satisfying $F_0 G_0 \neq 0$ is finite-dimensional, with dimension at most 10.
\end{theorem}

\begin{proof}
The first two equations in \eqref{bigsys} can be solved for $F_{1,u}$ and $G_{1,v}$ to obtain
\begin{equation}\label{solve-first-two-eqs}
F_{1,u} = -G_1 F_{1,v}, \qquad G_{1,v} = -F_1 G_{1,u}.
\end{equation}
These equations, together with their derivatives, determine all partial derivatives of $F_1$ involving any $u$-derivatives and all partial derivatives of $G_1$ involving any $v$-derivatives in terms of $v$-derivatives of $F_1$ and $u$-derivatives of $G_1$.

Next, observe that the last 6 equations in \eqref{bigsys} can be solved for $\alpha, \beta, \kappa, \eta, \mu, \xi$ to obtain
\begin{equation}\label{solve-for-greeks}
\begin{alignedat}{2}
\alpha & = \frac{(F_1 G_0)_v + F_1^2 G_{0,u}}{F_1 (1 - F_1 G_1)}, & \qquad \qquad
\beta & = \frac{(G_1 F_0)_u + G_1^2 F_{0,v}}{G_1 (1 - F_1 G_1)}, \\[0.1in]
\kappa & = \frac{F_{0,u} + G_1 F_{0,v} + F_1 F_0 G_{1,u}}{1 - F_1 G_1},  &\qquad \qquad
\eta & = \frac{G_{0,v} + F_1 G_{0,u} + G_1 G_0 F_{1,v}}{1 - F_1 G_1},  \\[0.1in]
\mu & = \frac{F_1 G_0 F_{0,v} - F_0 (F_1 G_0)_v}{F_1 (1 - F_1 G_1)}, &\qquad \qquad
\xi & = \frac{G_1 F_0 G_{0,u} - G_0 (G_1 F_0)_u}{G_1 (1 - F_1 G_1)}.
\end{alignedat}
\end{equation}
Differentiating equations \eqref{solve-for-greeks} and imposing the conditions
\[ \alpha_v = \beta_u = \kappa_v = \eta_u = \mu_v = \xi_v = 0 \]
yields 6 equations which can be solved algebraically for the second-order derivatives
\[ F_{1,vv},\, G_{1,uu},\, F_{0,uu},\, F_{0,uv},\, G_{0,uv},\, G_{0,vv}. \]
These expressions, together with their derivatives, determine all the remaining second and higher-order partial derivatives of $F_1$ and $G_1$, together with all second and higher-order partial derivatives of $F_0$ involving any $u$-derivatives and all second and higher-order partial derivatives of $G_0$ involving any $v$-derivatives.  For example, the expressions for the mixed partials of $F_0$ and $G_0$ take the form
\begin{align}
F_{0,uv} &= -\dfrac{F_0}{G_0 G_1} G_{0,uu} + \ldots, \label{foll}\\
G_{0,uv} &= -\dfrac{G_0}{F_0 F_1} F_{0,vv} + \ldots, \label{goll}
\end{align}
where we have suppressed terms involving only first-order partials.
Differentiating \eqref{foll} with respect to $v$ and \eqref{goll} with respect to $u$ yields a pair of equations
which are linear in $F_{0,uvu}$ and $G_{0,uvv}$ (and for which the coefficient matrix has nonzero determinant $(1-F_1 G_1)/(F_1 G_1)$), which
can be solved to express these third-order partials in terms of first-order partials and $F_{0,vv}, G_{0,uu}$.

We can also determine the third-order partials $F_{0,vvv}$ and $G_{0,uuu}$.  The expressions for the second-order partials we obtained before include
$$F_{0,uu} = \dfrac{F_0(1+F_1^2 G_1^2)}{G_0 F_1^2 G_1^2} G_{0,uu} - \dfrac{F_{0,vv}}{F_1^2} + \ldots, \qquad
G_{0,vv} = \dfrac{G_0(1+F_1^2 G_1^2)}{F_0 F_1^2 G_1^2} F_{0,vv} - \dfrac{G_{0,uu}}{G_1^2} + \ldots.$$
Computing $0=\left(F_{0,uu}\right)_v - \left(F_{0,uv}\right)_u$ using the first equation and \eqref{foll},
and computing $0=\left(G_{0,uv}\right)_v - \left(G_{0,vv}\right)_u$ using the second equation and \eqref{goll}, and using the
expressions for $F_{0,uvu}$ and $G_{0,uvv}$ obtained just above, yields a pair of equations which are linear in $F_{0,vvv},\, G_{0,uuu}$.
Since the coefficient matrix has a similar nonzero determinant, these can be solved to determine $F_{0,vvv}$ and $G_{0,uuu}$
in terms of first-order partials and $F_{0,vv}, G_{0,uu}$

This process produces a total differential system for the functions $F_0, G_0, F_1, G_1$ and a finite number of their derivatives.  In particular, the Frobenius theorem guarantees that the solution space has dimension at most 12, with any local solution being completely determined by the values of
\[ F_0,\, G_0,\, F_1,\, G_1,\, F_{0,u},\, F_{0,v},\, G_{0,u},\, G_{0,v},\, F_{1,v},\, G_{1,u},\, G_{0,uu},\, F_{0,vv} \]
at a single point $(u_0, v_0)$.  Moreover, these values are not independent; equating the mixed partial derivatives
\[ (F_{0,uvv})_v = (F_{0,vvv})_u, \qquad (G_{0,uuu})_v = (G_{0,uuv})_u \]
yields two functionally independent relations
that must be satisfied by these 12 parameters, resulting in a solution space of dimension at most 10.

\end{proof}

We note that it is possible---and indeed, likely---that differentiating the remaining relations alluded to above may yield more relations, thereby reducing the solution space even farther. Unfortunately the algebraic computations involved in doing so appear to be intractable, even with the assistance of {\sc Maple}.
However, we can make significant progress towards understanding the solution space in greater detail via direct analysis of the PDE system \eqref{bigsys}. In order to proceed in this direction, we must divide into cases based on whether the functions $F_1$ and $G_1$ are functionally dependent or independent.  First, we consider the case where these functions are both constant.

\subsection*{Case 1: $F_1, G_1$ both constant.}
In this case, we have the following proposition.

\begin{prop}\label{F1G1-constant-prop}
Suppose that $F_1, G_1$ are both constant.  Then
\begin{itemize}
\item If $F_1 G_1 \neq -1$, then $\B$ is a transformation between two constant-coefficient linear PDEs.
\item If $F_1 G_1 = -1$, then $\B$ is one of the following:
\begin{enumerate}
\item a transformation between two constant-coefficient linear PDEs;
\item a transformation between two PDEs of the form
\[ u_{xy} = a_1 e^{2ru} + a_2 e^{-2ru}, \qquad v_{xy} = b_1 e^{2rv} + b_2 e^{-2rv}; \]
\item a transformation between two PDEs of the form
\[ u_{xy} = a_1 \cos(2ru) + a_2 \sin(2ru), \qquad v_{xy} = b_1 \cos(2rv) + b_2 \sin(2rv). \]
\end{enumerate}
\end{itemize}
\end{prop}

\begin{proof}
From equations \eqref{secondline}, it follows that $\alpha(u) = \eta(v)$ and $\beta(v) = \kappa(u)$, and hence we must have
\[ \alpha(u) = \eta(v) = c_1, \qquad \beta(v) = \kappa(u) = c_2 \]
for some constants $c_1, c_2$.  Then equations \eqref{thirdline} become
\begin{equation}\label{case1-thirdline-reduced}
F_{0,u} + G_1 F_{0,v} = (1 - F_1 G_1) c_2, \qquad  G_{0,v} + F_1 G_{0,u}  =(1 - F_1 G_1) c_1.
\end{equation}
Since the right-hand sides of equations \eqref{case1-thirdline-reduced} are constants, it follows that
\begin{equation}\label{case1-thirdline-integrated}
F_0 = c_2 (u - F_1 v) + \phi(v - G_1 u), \qquad
G_0 = c_1 (v - F_1 u) + \psi(u - F_1 v)
\end{equation}
for some functions $\phi, \psi$ of one variable.

For ease of notation, let
\[ w = v - G_1 u, \qquad z = u - F_1 v, \]
so that we have
\begin{equation}\label{case1-F0G0-wz}
 F_0 = c_2 z + \phi(w), \qquad G_0 = c_1 w + \psi(z).
\end{equation}
Substituting \eqref{case1-F0G0-wz} into \eqref{fourthline} and solving for $\mu, \xi$ yields
\begin{equation}\label{case1-muxi-wz}
\begin{aligned}
 \mu & = -\frac{1}{(1 - F_1 G_1)} \left((c_2 z + \phi(w))(c_1 - F_1 \psi'(z)) + (c_1 w + \psi(z))(c_2 F_1 - \phi'(w)) \right), \\
 \xi & = -\frac{1}{(1 - F_1 G_1)} \left((c_2 z + \phi(w))(c_1 G_1 - \psi'(z)) + (c_1 w + \psi(z))(c_2 - G_1 \phi'(w)) \right).
\end{aligned}
\end{equation}
Now, $\mu$ must be a function of $u$ alone and $\xi$ must be a function of $v$ alone.  In terms of the $(w, z)$ coordinates, we have
\[ \frac{\partial}{\partial u} = \frac{\partial}{\partial z} - G_1 \frac{\partial}{\partial w}, \qquad
\frac{\partial}{\partial v} = \frac{\partial}{\partial w} - F_1 \frac{\partial}{\partial z}. \]
Therefore, we must have
\[ \frac{\partial}{\partial v} \mu = \left(\frac{\partial}{\partial w} - F_1 \frac{\partial}{\partial z}\right) \mu = 0, \qquad
 \frac{\partial}{\partial u} \xi = \left(\frac{\partial}{\partial z} - G_1 \frac{\partial}{\partial w}\right) \xi = 0 .\]
Applying these conditions to the expressions \eqref{case1-muxi-wz} yields the following differential equations for $\phi$ and $\psi$:
\begin{equation}\label{case1-phipsi-pdes}
\begin{aligned}
\phi''(w)(c_1 w + \psi(z)) & =  F_1^2 \psi''(z)(c_2 z + \phi(w)) , \\
G_1^2 \phi''(w)(c_1 w + \psi(z)) & = \psi''(z)(c_2 z + \phi(w))  .
\end{aligned}
\end{equation}

\begin{lemma}\label{case1-FG-not-minus-1}
If $F_1 G_1 \neq -1$, then $\B$ is a transformation between two constant-coefficient linear PDEs.
\end{lemma}

\begin{proof}
Since $F_1 G_1 \neq 1$, the hypothesis implies that $(F_1 G_1)^2 \neq 1$.  Then equations \eqref{case1-phipsi-pdes} imply that
\[ \phi''(w)(c_1 w + \psi(z)) = \psi''(z)(c_2 z + \phi(w)) = 0. \]
Note that these equations are equivalent to
\[ \phi''(w) G_0 = \psi''(z) F_0 = 0. \]
Since $F_0 G_0 \neq 0$, it follows that $\phi''(w) = \psi''(z) = 0$.  Thus $\phi(w)$ and $\psi(z)$ are linear functions of their arguments.  It follows immediately that $\mu$ is a linear function of $u$ and $\xi$ is a linear function of $v$.  Therefore, $\B$ is a transformation between two PDEs of the form
\[ u_{xy} = c_1 u_x + c_2 u_y + c_3 u + c_4, \qquad v_{xy} = c_1 v_x + c_2 v_y + c_5 v + c_6. \]
\end{proof}

Now suppose that $F_1 G_1 = -1$, so that equations \eqref{case1-phipsi-pdes} are linearly dependent.  By performing a change of variables of the form
\[ x \to \lambda x, \qquad y \to \frac{1}{\lambda} y \]
in exactly one of the underlying PDEs, we may assume that $F_1 = 1$, $G_1 = -1$; then equations \eqref{case1-phipsi-pdes} reduce to the single PDE
\begin{equation}\label{case1-phipsi-pde1}
\phi''(w)(c_1 w + \psi(z))  = \psi''(z)(c_2 z + \phi(w)).
\end{equation}
Note that this equation is equivalent to
$\phi''(w) G_0  = \psi''(z)F_0$, and since $F_0 G_0 \neq 0$, it follows that $\phi''(w), \phi''(z)$ are either both zero or both nonzero.  The argument in the proof of Lemma \ref{case1-FG-not-minus-1} shows that if $\phi''(w) = \psi''(z) = 0$, then $\B$ is a transformation between two constant-coefficient linear PDEs, so assume that $\phi''(w), \phi''(z)$ are both nonzero.

Differentiating equation \eqref{case1-phipsi-pde1} with respect to $z$ and $w$ yields
\begin{equation*}
\phi'''(w)\psi'(z)  = \psi'''(z)\phi'(w).
\end{equation*}
By hypothesis, $\phi'(w)$ and $\psi'(z)$ are not identically zero, so we may write this equation as
\begin{equation}\label{case1-phipsi-pde2}
\frac{\phi'''(w)}{\phi'(w)} = \frac{\psi'''(z)}{\psi'(z)},
\end{equation}
and both sides of equation \eqref{case1-phipsi-pde2} must be constant.

{\bf Case 1(a):}
Both sides of equation \eqref{case1-phipsi-pde2} are equal to zero.  Then we have
\[ \phi(w) = a_2 w^2 + a_1 w + a_0, \qquad \psi(z) = b_2 z^2 + b_1 z + b_0 \]
for some constants $a_i, b_i$.  Substituting into equation \eqref{case1-phipsi-pde1} yields
\begin{equation}\label{case1-phipsi-pde1-linear}
 a_2 b_2 (w^2 - z^2) + (a_1 b_2 - a_2 c_1) w + (b_2 c_2 - a_2 b_1) z + (b_2 a_0 - a_2 b_0) = 0.
\end{equation}
From the leading term, we have $a_2 b_2 = 0$.  If $a_2 = b_2 = 0$, then $\phi(w)$ and $\psi(z)$ are both linear functions, and $\B$ is a transformation between two constant-coefficient linear PDEs.  So suppose that $a_2 \neq 0$.  (The case $b_2 \neq 0$ is analogous.) Then equation \eqref{case1-phipsi-pde1-linear} implies that $b_2 = b_1 = b_0 = c_1 = 0$,  but then $G_0 = 0$, contrary to hypothesis.

{\bf Case 1(b):}
Both sides of equation \eqref{case1-phipsi-pde2} are equal to a positive constant $r^2$.  Then we have
\[ \phi(w) = a_1 e^{rw} + a_2 e^{-rw} + a_0, \qquad \psi(z) = b_1 e^{rz} + b_2 e^{-rz} + b_0 \]
for some constants $a_i, b_i$. Substituting into equation \eqref{case1-phipsi-pde1} yields
\begin{equation}\label{case1-phipsi-pde1-exp}
(c_1 w + b_0) (a_1 e^{rw} + a_2 e^{-rw}) - (c_2 z + a_0)(b_1 e^{rz} + b_2 e^{-rz}) = 0.
\end{equation}
If $a_1 = a_2 = b_1 = b_2 = 0$, then $\phi(w)$ and $\psi(z)$ are both constant functions, and $\B$ is a transformation between two constant-coefficient linear PDEs.  So suppose that at least one of $a_1, a_2$ is nonzero.  (The assumption that at least one of $b_1, b_2$ is nonzero is similar.)  Then equation \eqref{case1-phipsi-pde1-exp} implies that $b_0 = c_1 = 0$.  If we also have $b_1 = b_2 = 0$, then $G_0 = 0$, contrary to hypothesis; therefore, it must also be the case that at least one of $b_1, b_2$ is nonzero. Then equation \eqref{case1-phipsi-pde1-exp} implies that $a_0 = b_0 = c_1 = c_2 = 0$.  Returning to $(u,v)$ coordinates, it follows that
\[ F_0 = a_1 e^{r(u+v)} + a_2 e^{-r(u+v)}, \qquad G_0 = b_1 e^{r(u-v)} + b_2 e^{-r(u-v)}, \]
and $\B$ is a transformation between two PDEs of the form
\[ u_{xy} = r(a_1 b_1 e^{2ru} - a_2 b_2 e^{-2ru}), \qquad v_{xy} = -r(a_1 b_2 e^{2rv} - a_2 b_1 e^{-2rv}). \]
Relabeling the coefficients gives the form in Proposition \ref{F1G1-constant-prop}.

{\bf Case 1(c):}
Both sides of equation \eqref{case1-phipsi-pde2} are equal to a negative constant $-r^2$.  Then we have
\[ \phi(w) = a_1 \cos(rw) + a_2 \sin(rw) + a_0, \qquad \psi(z) = b_1 \cos(rz) + b_2 \sin(rz) + b_0 \]
for some constants $a_i, b_i$. Substituting into equation \eqref{case1-phipsi-pde1} yields
\begin{equation}\label{case1-phipsi-pde1-trig}
(c_1 w + b_0) (a_1 \cos(rw) + a_2 \sin(rw)) - (c_2 z + a_0)(b_1 \cos(rz) + b_2 \sin(rz)) = 0.
\end{equation}
The same argument as in the previous case shows that $a_0 = b_0 = c_1 = c_2 = 0$. Returning to $(u,v)$ coordinates, it follows that
\[ F_0 = a_1 \cos(r(u+v)) + a_2 \sin(r(u+v)), \qquad G_0 = b_1 \cos(r(u-v)) + b_2 \sin(r(u-v)), \]
and $\B$ is a transformation between two PDEs of the form
\begin{gather*}
u_{xy} = \tfrac{1}{2}r\big{[}(a_1 b_2 + a_2 b_1)\cos(2ru) + (a_2 b_2 - a_1 b_1)\sin(2ru)\big{]}, \\
v_{xy} = \tfrac{1}{2}r\big{[}(a_1 b_2 - a_2 b_1)\cos(2rv) + (a_2 b_2 + a_1 b_1)\sin(2rv)\big{]}.
\end{gather*}
Relabeling the coefficients gives the form in Proposition \ref{F1G1-constant-prop}.
\end{proof}

As a consequence of the above, we note that when $F_1, G_1$ are both constant $\B$ is either a transformation between two constant-coefficient linear PDEs or a transformation between two $f$-Gordon equations of the form
\begin{equation}\label{case1-f-Gordon}
u_{xy} = f(u), \qquad v_{xy} = g(v).
\end{equation}
Conversely, equations \eqref{topline}, \eqref{secondline} imply that for any $f$-Gordon equations with $f(u)$, $g(v)$ not both zero (which implies that $F_0 G_0 \neq 0$), both $F_1$ and $G_1$ must be constant.  This means that the only B\"acklund transformations of this form between $f$-Gordon equations are those specified in Proposition \ref{F1G1-constant-prop}.  Thus we have the following Proposition.

\begin{prop}\label{autonomous-f-Gordon-prop}
Suppose that $\B$ is a quasilinear, wavelike, autonomous normal B\"acklund  transformation between two $f$-Gordon equations $u_{xy} = f(u)$ and $v_{xy} = g(v)$. Then $\B$ is one of the transformations specified in Proposition \ref{F1G1-constant-prop}.  In particular, the functions $f(u)$ and $g(v)$ must have one of the following forms:
\begin{enumerate}
\item $f(u) = a_1 u + a_0$, $g(v) = b_1 v + b_0$;
\item $f(u) = a_1 e^{2ru} + a_2 e^{-2ru}$, $g(v) = b_1 e^{2rv} + b_2 e^{-2rv}$;
\item $f(u) = a_1 \cos(2ru) + a_2 \sin(2ru)$, $g(v) = b_1 \cos(2rv) + b_2 \sin(2rv)$.
\end{enumerate}
\end{prop}
We also note that this proposition does not preclude the possibility that other $f$-Gordon equations may have B\"acklund transformations satisfying less stringent conditions.


\bigskip
For the rest of this section, we will assume that at least one of $F_1, G_1$ is non-constant. This naturally
leads to two possibilities, where the differentials $dF_1, dG_1$ are either linearly dependent or independent at each point.  We label these as Cases 2 and 3, respectively.
\bigskip

\subsection*{Case 2: $F_1, G_1$ dependent.}

\begin{prop}\label{F1G1-dependent-prop}
Suppose that $F_1, G_1$ are functionally dependent, i.e., that $dF_1 \wedge dG_1 = 0$, but $F_1$ and $G_1$ are not both constant.  Then one of $F_1, G_1$ is constant, and (assuming WLOG that $G_1$ is constant) $\B$ is one of the following:
\begin{enumerate}
\item a transformation between two constant-coefficient linear PDEs;
\item a transformation between either two PDEs of the form
\[ u_{xy} = c_3 e^{G_1 r_2 u}, \qquad v_{xy} = c_2 e^{r_2 v} v_y, \]
or two PDEs of the form
\[ u_{xy} = c_1 e^{-G_1 r_2 u} u_x, \qquad v_{xy} = c_4 e^{-r_2 v}; \]
\item a transformation between two PDEs of the form
\[ u_{xy} = c_{1} e^{G_1 r_0 u} u_x, \qquad v_{xy} = c_{2} e^{r_0 v} v_y; \]
\item a transformation between two PDEs of the form
\[ u_{xy} = c_1 e^{G_1 r_0 u} \left( G_1 u_x - \frac{b}{r_0} \right) + b u_y , \qquad
v_{xy} = c_2 e^{r_0 v} \left( v_y - \frac{b}{r_0} \right) + b v_x
\]
with $b, c_1, c_2 \neq 0$.
\end{enumerate}

\end{prop}

\begin{remark}
With the exception of the last transformation on the list, these are all transformations involving either Liouville's equation $z_{xy} = e^z$, the Monge-integrable equation $z_{xy} = e^z z_x$ (or equivalently, $z_{xy} = e^z z_y$), or constant coefficient linear equations, all of which appear on Goursat's list \cite{Goursat} of PDEs that are Darboux-integrable at second order.  The transformation (4), however, is different: Up to scalings, translations, and interchanging the independent variables, both underlying PDEs are equivalent to the PDE
\begin{equation}\label{case2-new-example}
 z_{xy} = e^z (z_y - b) + b z_x.
\end{equation}
This PDE is Monge-integrable: For any solution $z(x,y)$, the quantity $I = e^{-by}(z_x - e^z)$ satisfies $\frac{\partial I}{\partial y} = 0$.  However, the PDE \eqref{case2-new-example} is not Darboux-integrable at second order unless $b=0$.
\end{remark}

\begin{proof}[Proof of Prop. 4.5]
First we prove the following lemma:

\begin{lemma} If $dF_1 \wedge dG_1 = 0$, then either $F_1$ or $G_1$ is locally constant.
\end{lemma}
\begin{proof}
We may rewrite \eqref{topline} as
$$\begin{bmatrix} F_{1,u} & F_{1,v}\end{bmatrix}\cdot \begin{bmatrix} 1 & G_1\end{bmatrix}=0,\qquad
\begin{bmatrix} G_{1,u} & G_{1,v}\end{bmatrix}\cdot \begin{bmatrix} F_1 & 1\end{bmatrix}=0.
$$
Since the vectors $\begin{bmatrix} F_{1,u} & F_{1,v}\end{bmatrix}$ and $\begin{bmatrix} G_{1,u} & G_{1,v}\end{bmatrix}$
are linearly dependent, but the vectors
$\begin{bmatrix} 1 & G_1\end{bmatrix}$ and $\begin{bmatrix} F_1 & 1\end{bmatrix}$ are linearly independent, it follows that one of the vectors $\begin{bmatrix} F_{1,u} & F_{1,v}\end{bmatrix}$, $\begin{bmatrix} G_{1,u} & G_{1,v}\end{bmatrix}$ must be zero.  Therefore, one of $F_1$ or $G_1$ must be locally constant.
\end{proof}

Without loss of generality, {\bf we will assume that $G_1$ is constant} in the rest of this subsection.
From \eqref{secondline}, we have $\beta(v) = \kappa(u)$, and hence we must have
$$\beta=\kappa=b$$
for some constant $b$.
Moreover, \eqref{topline} implies that $F_1$ is a function of the single variable $v - G_1 u$.  Thus, if we let $w = v - G_1 u$, we may set\begin{equation}\label{F1twoform}
F_1 = \varphi(w)
\end{equation}
for some unknown function $\varphi(w)$. Similarly, the first equation in \eqref{thirdline} gives
\begin{equation}\label{F0twoform}
\begin{aligned}
F_0 & = \psi(w) + b (1-F_1 G_1) u \\
& = \psi(w) + b (1- G_1 \varphi(w)) u
\end{aligned}
\end{equation}
for some unknown function $\psi(w)$.  Then solving \eqref{secondline} for $G_0$ gives
\begin{equation}\label{G0twoform}
G_0 = \dfrac{\varphi}{\varphi'}(\alpha-\eta).
\end{equation}
Thus, the \BT/s in this case are determined by the single-variable functions $\alpha(u)$, $\eta(v)$, $\mu(u)$, $\xi(v)$,
$\varphi(w)$ and $\psi(w)$.

Substituting \eqref{G0twoform} into the second equation of \eqref{thirdline} and dividing by the (nonzero) coefficient of $\eta_v$ gives
\begin{equation}\label{zero3}
(\alpha_u + A_0 \alpha)\varphi = \eta_v + B_0 \eta,
\end{equation}
where
\begin{equation}\label{AABB}
A_0 = \frac{1}{\varphi^2 \varphi'} ((\varphi')^2 - (1 - G_1\varphi) \varphi \varphi''), \quad
B_0 = \frac{1}{\varphi \varphi'} ((\varphi')^2 - (1 - G_1\varphi)(\varphi \varphi'' - (\varphi')^2)).
\end{equation}
In what follows, we will use this equation, together with the fact that $\alpha$, $\eta$ and $\varphi$ are functions of
different variables, to narrow down the possibilities.

\begin{lemma}\label{A0B0-both-constant-lemma}
The functions $A_0$ and $B_0$ are either both constant or both non-constant.
\end{lemma}
\begin{proof}  It is easy to check that $d B_0/dw = \varphi\, d A_0/dw$.
\end{proof}

\begin{prop}\label{third-order-ODE-prop}
The functions $\alpha$ and $\eta$ each satisfy constant-coefficient homogeneous linear differential equations of
order at most $3$.  The equations are of order at most 2 if $A_0,B_0$ are constants; in that case if
$\varphi'/\varphi$ is non-constant then $\alpha$ and $\eta$ are constant multiples of $e^{r_1 u}$ and $e^{r_2 v}$ respectively,
where $A_0 =-r_1$ and $B_0 = -r_2$.
\end{prop}
\begin{proof}
First, suppose that $A_0 = -r_1$ and $B_0 = -r_2$ are constant.  Differentiating equation \eqref{zero3} with respect to $u$ and then dividing by $\varphi$  yields a second-order linear ODE
\begin{equation}\label{case2-alpha-ODE-A0B0const}
 \alpha_{uu} - \left(r_1 + G_1 \frac{\varphi'}{\varphi}\right) \alpha_u + G_1 r_1 \frac{\varphi'}{\varphi}\alpha = 0
\end{equation}
where the coefficients are functions of $w=v- G_1 u$.  In fact, the roots of the auxiliary equation are $r_1$ and $\rho = G_1 \varphi'/\varphi$, so that
$$\alpha = C_1 e^{r_1 u} + C_2 e^{\rho u},$$
where $\rho, C_1, C_2$ are {\it a priori} functions of $w$.
However, the fact that $\alpha$ must have no $w$-dependence forces $C_1, C_2$ to be constants, as well as
$\rho$ if $C_2 \ne 0$.  (To see this, set the derivative of the expression for $\alpha$ with respect to $w$, holding $u$ fixed, equal to zero; then take a $u$-derivative, holding $w$ fixed.  This results in two homogeneous linear equations in the exponentials, and thus the matrix of coefficients for these equations must be singular.  Then setting the highest-order coefficient of $u$ in the
determinant yields this result.)  Thus, $\alpha$ satisfies a constant-coefficient homogeneous linear ODE of order at most two, and if $\rho$ is non-constant $\alpha$ must be a constant multiple of $e^{r_1 u}$.  Similarly, dividing equation \eqref{zero3} by $\varphi$, differentiating with respect to $v$, and then multiplying by $\varphi$ yields \begin{equation}\label{case2-eta-ODE-A0B0const}
\eta_{vv} - \left(r_2 + \frac{\varphi'}{\varphi}\right) \eta_v + r_2 \frac{\varphi'}{\varphi} \eta = 0.
\end{equation}
Now we apply the same, argument noting here that the roots of the auxiliary equation are $r_2$  and $\varphi'/\varphi$.

When $A_0$ and $B_0$ are non-constant functions of $w$, the argument is similar, but with an additional step. Differentiating \eqref{zero3} with respect to $u$ (holding $v$ fixed) yields
\[ \varphi\, \alpha_{uu}  + (-G_1 \varphi' + \varphi A'_0) \alpha_u - G_1 (\varphi A_0)' \alpha = -G_1 B'_0 \eta, \]
where primes denote derivatives with respect to $w$.  Since $B_0$ is assumed to be non-constant, we may divide both sides by $B'_0$ and differentiate again with respect to $u$ to eliminate the $\eta$ term.
This results in a third-order ODE for the function $\alpha(u)$, with coefficients that are functions of $w$.  Thus, $\alpha$ is a sum of
terms of the form $C_i e^{\rho_i u}$ where $C_i,\rho_i$ are (possibly complex-valued) functions of $w$. However, a similar procedure of taking a $w$-derivative
(holding $u$ fixed) and two additional $u$-derivatives (holding $w$ fixed) results in a set of linear homogeneous equations which are
admit non-trivial solutions only if each $C_i$ is constant, and $\rho_i$ is constant whenever $C_i$ is nonzero.  Thus, $\alpha$ satisfies
a constant-coefficient homogeneous linear ODE of order at most three.  A similar argument yields an ODE of order at most three satisfied
by $\eta(v)$.
\end{proof}

We now further pursue each of the cases indicated by Prop. \ref{third-order-ODE-prop}.

\medskip
\noindent
{\bf Case 2(a):} Suppose that $A_0 = -r_1$ and $B_0 = -r_2$ are constant.

First, suppose that $\varphi'/\varphi = -r_0$ is constant.
Then we have $\varphi(w) = c_0 e^{-r_0 w}$ for some nonzero constant $c_0$.  (Note that, since we have assumed $\varphi' \neq 0$, we must have $r_0 \neq 0$.)  Then the expressions \eqref{AABB} and the conditions $A_0 = -r_1$, $B_0 = -r_2$ imply that
\[ r_1 = G_1 r_0, \qquad r_2 = r_0. \]
Conversely, if $r_1 = G_1 r_2$, then the general solution to the (compatible) ODEs for $\varphi$ determined by \eqref{AABB} and the conditions $A_0 = -r_1$, $B_0 = -r_2$ is $\varphi(w) = c_0 e^{-r_0 w}$.  In this case, the general solution to equation \eqref{case2-alpha-ODE-A0B0const} is
\[ \alpha(u) = c_{11} e^{G_1 r_0 u} + c_{12} e^{-G_1 r_0 u}, \]
and the general solution to equation \eqref{case2-eta-ODE-A0B0const} is
\[ \eta(v) = c_{21} e^{r_0 v} + c_{22} e^{-r_0 v}. \]
Substituting these expressions, together with $\varphi = c_0 e^{-r_0(v - G_1u)}$, into \eqref{zero3} yields
\[ c_{22} =  G_1 c_0 c_{12}. \]
Now equations \eqref{fourthline} can be solved to obtain expressions for $\mu$ and $\xi$.
Then the conditions that $\mu$ is a function of $u$ alone and $\xi$ is a function of $v$ alone imply (after substantial computation) that we have the following possibilities:
\begin{enumerate}[label=(\roman*)]
\item $b \neq 0$,\  $c_{12} = c_{11} - G_1 c_0 c_{21} = 0$, and $\psi(w)$ is a solution to the ODE
\[ \psi' = -\frac{G_1 c_0 r_0}{(G_1 c_0 - e^{r_0 w})} \psi - b c_0 e^{-r_0 w}. \]
Then we have $\mu = \displaystyle{-\frac{b c_0 c_{21}}{r_0} e^{G_1 r_0 u}}$,\  $\xi = \displaystyle{-\frac{b c_{21}}{r_0} e^{r_0 v}}$, and $\B$ is a transformation between equations of the form
\[ u_{xy} = c_0 c_{21} e^{G_1 r_0 u} \left( G_1 u_x - \frac{b}{r_0} \right) + b u_y , \qquad
v_{xy} = c_{21} e^{r_0 v} \left( v_x - \frac{b}{r_0} \right) + b v_y.
\]
\item $b = c_{12} = \mu = \xi = 0$, \ $\psi = c_3(c_{21} - c_{11} e^{-r_0 w})$, and $\B$ is a transformation between equations of the form
\[ u_{xy} = c_{11} e^{G_1 r_0 u} u_x, \qquad v_{xy} = c_{21} e^{r_0 v} v_x. \]  
\end{enumerate}

Suppose that $\varphi'/\varphi$ is not constant---or equivalently, that $r_1 \neq G_1 r_2$.  Then the only solutions to equation \eqref{case2-alpha-ODE-A0B0const} are
$\alpha = c_{1} e^{r_1 u}$ and the only solutions to equation \eqref{case2-eta-ODE-A0B0const} are $\eta = c_{2} e^{r_2 v}$.  Furthermore, the ODEs for $\varphi$ given by $A_0 = -r_1$, $B_0 = -r_2$ are compatible and may be reduced to the single separable first-order ODE
\begin{equation}\label{case2-phi-ODE-A0B0const}
 \varphi' = \frac{\varphi (r_2 - r_1 \varphi)}{G_1 \varphi - 1},
\end{equation}
whose solution may be given as an implicitly defined function of $w$ but is not particularly enlightening.
Now equations \eqref{fourthline} can be solved to obtain expressions for $\mu$ and $\xi$.  Then the conditions that $\mu$ is a function of $u$ alone and $\xi$ is a function of $v$ alone imply (after substantial computation) that we have the following possibilities:
\begin{enumerate}[resume*]
\item $b = c_1 = 0$, $r_1 = -G_1 r_2$, and $\psi(w)$ is a solution to the ODE
\[ \psi' = \frac{2 G_1 r_2 \varphi}{(1 - G_1 \varphi)^2} \psi. \]
Then we have $\mu = c_3 e^{G_1 r_2 u}$, $\xi = 0$, and $\B$ is a transformation between equations of the form
\[ u_{xy} = c_3 e^{G_1 r_2 u}, \qquad v_{xy} = c_2 e^{r_2 v} v_x. \]  
\item $b = c_2 = 0$, $r_1 = -G_1 r_2$, and $\psi$ is a solution to the ODE
\[ \psi' = \frac{r_2 (G_1^2 \varphi^2 + 2 G_1 \varphi - 1)}{(1 - G_1 \varphi)^2} \psi. \]
Then we have $\mu = 0$, $\xi = c_4 e^{-r_2 v}$, and $\B$ is a transformation between equations of the form
\[ u_{xy} = c_1 e^{-G_1 r_2 u} u_x, \qquad v_{xy} = c_4 e^{-r_2 v}. \]
\item Either $r_1 = c_2 = 0$ and $r_2 c_1 \neq 0$, or $r_2 = c_1 = 0$ and $r_1 c_2 \neq 0$.  In this case, $\B$ is a transformation between constant-coefficient linear PDEs.
\end{enumerate}

{\bf Case 2(b):} Suppose that $A_0$ and $B_0$ are non-constant functions of $w$. By Proposition \ref{third-order-ODE-prop}, the functions $\alpha(u)$ and $\eta(v)$ are each solutions of a constant-coefficient homogeneous linear ODE of order at most 3.  This implies that $\alpha$ has one of of the following forms (allowing for the possibility of complex values for the constants $r_{ij}$ and $c_{ij}$):
\begin{enumerate}
\item $\alpha(u) = c_{11} e^{r_{11}u} + c_{12} e^{r_{12}u} + c_{13} e^{r_{13}u}$, with $r_{11}, r_{12}, r_{23}$ all distinct;
\item $\alpha(u) = c_{11} e^{r_{11}u} + e^{r_{12}u}(c_{12} + c_{13} u)$, with $r_{11}, r_{12}$ distinct;
\item $\alpha(u) = e^{r_{11}u} (c_{11} + c_{12} u + c_{13} u^2)$.
\end{enumerate}
Similarly, $\eta$ has one of the forms:
\begin{enumerate}
\item $\eta(v) = c_{21} e^{r_{21}v} + c_{22} e^{r_{22}v} + c_{23} e^{r_{23}v}$, with $r_{21}, r_{22}, r_{23}$ all distinct;
\item $\eta(v) = c_{21} e^{r_{21}v} + e^{r_{22}v}(c_{22} + c_{23} v)$, with $r_{21}, r_{22}$ distinct;
\item $\eta(v) = e^{r_{21}v} (c_{21} + c_{22} v + c_{23} v^2)$.
\end{enumerate}
Moreover, substituting these expressions into \eqref{zero3} and evaluating along any line of the form $v - G_1 u = C$, with $C$ constant and chosen so that neither $-A_0$ nor $-B_0$ is equal to any of the constants $r_{ij}$ along this line, shows that $\alpha$ and $\eta$ must both be of the same form, with $r_{1j} = G_1 r_{2j}$ for all $j$ and with $c_{1j} = 0$ if and only if $c_{2j}=0$.

For each of the possible forms for the pair $(\alpha(u), \eta(v))$, substituting $v = G_1 u + w$ into equation \eqref{zero3} and comparing like terms in $u$ imposes one or more ODEs on the function $\varphi(w)$.
In the cases where these ODEs are compatible, solving equations \eqref{fourthline} for $\mu$ and $\xi$, and then imposing the additional conditions that $\mu$ is a function of $u$ alone and $\xi$ is a function of $v$ alone implies (after substantial computation) that either $\psi = b = 0$ (which contradicts the hypothesis $F_0 \neq 0)$ or the underlying PDEs for $u$ and $v$ are both constant-coefficient linear equations.

\end{proof}

\bigskip

\subsection*{Case 3: $F_1, G_1$ independent.}

The following lemma, whose proof is an easy computation, will be key to our analysis of this case:

\begin{lemma}\label{exactdhdk}  The equations \eqref{topline} imply that the following 1-forms are closed and hence locally exact:
\[ dh = \dfrac1\Delta (dv - G_1 du), \qquad dk = \dfrac1\Delta (du - F_1 dv), \]
where $\Delta = 1-F_1 G_1$.
\end{lemma}
The differentials $dh, dk$ are linearly independent, and thus we may introduce local coordinates $(h,k)$ as a local alternative to $(u,v)$.  The differentials $du, dv$ are related to $dh, dk$ by
\begin{equation}\label{duv to dhk}
du = dk + F_1 dh, \qquad dv = dh + G_1 dk.
\end{equation}
For use below, we note that the coordinate vector fields in the two systems are related by
\begin{equation}\label{deehkdefs}
\begin{aligned}
\dfrac{\di}{\di h} &= \dfrac{\di}{\di v} + F_1 \dfrac{\di}{\di u}  & \dfrac{\di}{\di k} &= \dfrac{\di}{\di u} + G_1 \dfrac{\di}{\di v},  \\
\dfrac{\di}{\di u} &= \dfrac{1}\Delta\left( \dfrac{\di}{\di k} - G_1 \dfrac{\di}{\di h}\right), &
\dfrac{\di}{\di v} &= \dfrac{1}\Delta\left( \dfrac{\di}{\di h} - F_1 \dfrac{\di}{\di k}\right) .
\end{aligned}
\end{equation}
From this, is it evident that equations \eqref{topline} are equivalent to the conditions that $F_1$ is a function of $h$ alone and $G_1$ is a function of $k$ alone.

Let $F_1', F_1''$, etc., denote the derivatives of $F_1$
with respect to $h$, and similarly for the derivatives of $G_1$ with respect to $k$.  Then from \eqref{deehkdefs} we have
\[ F_{1,v} = \dfrac{1}{\Delta} F_1', \qquad G_{1,u} = \dfrac{1}{\Delta} G_1'. \]
Taking this into account, solving \eqref{secondline} for $F_0$ and $G_0$ gives
\begin{equation}\label{fogovals}
F_0 = (\beta-\kappa) \Delta \dfrac{G_1}{G_1'}, \qquad G_0 = (\alpha-\eta) \Delta \dfrac{F_1}{F_1'}.
\end{equation}
Equations \eqref{thirdline} may now be written as
\[ \dfrac{\di}{\di k} F_0 = \kappa - F_1 G_1 \beta, \qquad \dfrac{\di}{\di h} G_0= \eta - F_1 G_1 \alpha. \]
Substituting the expressions \eqref{fogovals} into these equations gives (after some simplifications)
\begin{align}
F_1\,(\alpha_u + A_0 \alpha) = \eta_v + B_0 \eta,
 \label{zero3ap}\\
G_1\, (\beta_v + C_0 \beta) = \kappa_u + D_0 \kappa,
\label{zero4ap}
\end{align}
where
\begin{equation}\label{A0B0C0D0-expressions}
\begin{alignedat}{2}
A_0 & = \frac{ (F_1')^2 - F_1 F_1''}{F_1^2 F_1'}, & \qquad B_0 & = \frac{2 (F_1')^2 - F_1 F_1''}{F_1 F_1'}, \\
C_0 & = \frac{ (G_1')^2 - G_1 G_1''}{G_1^2 G_1'}, & \qquad D_0 & = \frac{2 (G_1')^2 - G_1 G_1''}{G_1 G_1'}.
\end{alignedat}
\end{equation}

\begin{lemma}\label{A0B0C0D0-constant-lemma}
The functions $A_0$ and $B_0$ (resp., $C_0$ and $D_0$) are either both constant or both non-constant.
\end{lemma}
\begin{proof}
The conclusion follows immediately from the identities $\dfrac{d B_0}{dh} = F_1 \dfrac{d A_0}{dh}$ and $\dfrac{d D_0}{dk} = G_1 \dfrac{d C_0}{dk}$.
We can also reach this conclusion by solving the ODEs obtained by setting one of these coefficients to a constant (and we will need these solutions later anyway).  For example, setting $A_0 = -m_1$ is     to the ODE
\begin{equation}\label{A0-const-ODE}
 F_1 F_1'' - (F_1')^2 = m_1 F_1^2 F_1'
\end{equation}
which has a first integral $\dfrac{F_1'}{F_1} = m_1 F_1 - n_1$ for a constant $n_1$.  Using this and the identity $B_0 = F_1 A_0 + \dfrac{F_1'}{F_1}$ immediately gives $B_0 = -n_1$.  Further integration gives
$F_1 = \dfrac{n_1}{q_1 e^{n_1 h} + m_1}$
if $n_1 \neq 0$, or $F_1 = \dfrac{1}{q_1 - m_1 h}$ when $n_1=0$.  (Because $F_1' \ne 0$, in the first case $q_1 \ne 0$ and in the second case $m_1 \ne 0$.)
%
%
\end{proof}

The last lemma shows that we can hope to determine transformations explicitly in the following special case:

\medskip
\noindent
{\bf Case 3(a):} Suppose that $A_0 = -m_1$, $B_0 = -n_1$, $C_0 = -m_2$, and $D_0 = -n_2$ are all constant.

\begin{prop}\label{F1G1-independent-A0B0C0D0-const-prop}
Suppose that $F_1, G_1$ are functionally independent and that the functions $A_0, B_0, C_0, D_0$ in equations \eqref{A0B0C0D0-expressions} are all constant.
Then $\B$ is one of the following:
\begin{enumerate}
\item a transformation between two PDEs of the form
\[   u_{xy} = s_1 e^{mu} u_x + t_2 e^{-mu} u_y, \qquad v_{xy} = s_2 e^{-nv} v_x + t_1 e^{nv} v_y   \]
with $m, n \neq 0$ and $s_1 t_2 = s_2 t_1$;
\item a transformation between a PDE of the form
\[  u_{xy} = s_1 e^{mu} u_x + t_2 e^{-mu} u_y   \]
with $m \neq 0$, and a constant-coefficient linear PDE.
\end{enumerate}

\end{prop}

\begin{proof}
First, we note that  equations \eqref{zero3ap}, \eqref{zero4ap} have the obvious solutions
\begin{equation}\label{obvious-greek-solns}
\alpha(u) = s_1 e^{m_1 u}, \qquad \eta(v) = s_2 e^{n_1 v}, \qquad \beta(v) = t_1 e^{m_2 v}, \qquad \kappa(u) = t_2 e^{n_2 u},
\end{equation}
where $s_1, s_2, t_1, t_2 \in \R$.  These are precisely the solutions for which both sides of equations \eqref{zero3ap} and \eqref{zero4ap} vanish identically.  Next, we will consider the other possibilities.

\begin{lemma}\label{more-solutions-lemma}
With $A_0, B_0, C_0, D_0$ as above, equations \eqref{zero3ap}, \eqref{zero4ap} have no additional solutions besides those in \eqref{obvious-greek-solns} unless
\[ m_2 = -n_1, \qquad n_2 = -m_1. \]

\end{lemma}

\begin{proof}
Without loss of generality, suppose that equation \eqref{zero3ap} has
additional solutions besides those in \eqref{obvious-greek-solns}. For these solutions, neither side of \eqref{zero3ap} is identically zero, and consequently
\begin{equation}\label{F1-is-a-product}
 F_1 = \frac{U(u)}{V(v)}
\end{equation}
for some functions $U, V$.  (Specifically, $U = (\alpha_u - m_1 \alpha)^{-1}$ and $V = (\eta_v - n_1 \eta)^{-1}$.)  Substituting this expression into the first equation from \eqref{topline} yields
\begin{equation}\label{G1-is-a-product}
 G_1 = \frac{U' V}{U V'},
\end{equation}
and then the second equation in \eqref{topline} is equivalent to
\[ \frac{U U'' - (U')^2}{U'} = \frac{V V'' - (V')^2}{V'}. \]
Since the left-hand side is a function of $u$ alone and the right-hand side is a function of $v$ alone, both sides must be equal to a constant $c_0$.
Setting both sides equal to zero, solving for $U''$ and $V''$, using these expressions to determine $F_1''$ and $G_1''$, and finally substituting
in \eqref{A0B0C0D0-expressions} and simplifying, yields
$$D_0 = -A_0 = \dfrac{U' + c_0}{U}, \qquad B_0 = -C_0 = \dfrac{V'+c_0}{V}.$$
%
%
%
\end{proof}

In light of Lemma \ref{more-solutions-lemma}, we first consider the cases where $m_2 = -n_1$ and $n_2 = -m_1$.  For ease of notation, set
\[ m = m_1 = -n_2,  \qquad  n = m_2 = -n_1. \]
We note that $m,n$ cannot both be zero, since $B_0 - F_1 A_0 = F_1'/F_1$ and so $A_0=B_0=0$ would imply that $F_1$ is constant.

\medskip
\noindent
{\em Case 3(a)(i):} Suppose that $m, n$ are both nonzero.  Then $A_0 = -m$ and $C_0 = -n$ give second-order ODEs for $F_1, G_1$ that
we solve, as in the proof of Lemma \ref{A0B0C0D0-constant-lemma}, to obtain
\begin{equation}\label{F1G1-in-terms-of-hk-mn-nonzero}
 F_1(h) = \frac{n}{q e^{-n h} - m}, \qquad G_1(k) = \frac{m}{r e^{-m k} - n}
\end{equation}
for nonzero constants $q, r \in \R$.  Then we can integrate equations \eqref{duv to dhk} to obtain the local coordinate transformation
\begin{equation}\label{hk-to-uv-nice-mn}
u = k - \frac{1}{m}(nh + \ln(q e^{-nh} - m)), \qquad
v = h - \frac{1}{n}(mk + \ln(r e^{-mk} - n)).
\end{equation}
The inverse transformation is given by
\begin{equation}\label{uv-to-hk-nice-mn}
h = v - \frac{1}{n} \ln \left( \dfrac{m n e^{mu + nv} - 1}{q n e^{mu} - r} \right)  , \qquad
k = u - \frac{1}{m} \ln \left( \dfrac{m n e^{mu + nv} - 1}{r m e^{nv} - q} \right) ,
\end{equation}
and thus we may write the expressions \eqref{F1G1-in-terms-of-hk-mn-nonzero} in terms of the $(u,v)$ coordinates as
\begin{equation}\label{F1G1-in-terms-of-uv-mn-nonzero}
F_1 = \frac{n(r - qn e^{mu})}{q e^{-nv} - rm}, \qquad
G_1 = \frac{m(q - rm e^{nv})}{r e^{-mu} - qn}.
\end{equation}
Now, after clearing denominators appropriately, equations \eqref{zero3ap}, \eqref{zero4ap} may be written as
\begin{align}
n (q n e^{mu} - r)(\alpha_u - m \alpha) & = (r m - q^{-nv})(\eta_v + n \eta), \label{eqns-for-greeks-nice-mn-1} \\
m (r m e^{nv} - q)(\beta_v - n \beta) & = (q n - r^{-mu})(\kappa_u + m \kappa).  \label{eqns-for-greeks-nice-mn-2}
\end{align}
In each of these equations, one side is a function of $u$ alone and the other side is a function of $v$ alone, so both sides must be constant.  A straightforward computation shows that the general solution of equations \eqref{eqns-for-greeks-nice-mn-1} and \eqref{eqns-for-greeks-nice-mn-2} is given by
\begin{equation}\label{not-so-obvious-greek-solns-mn-nonzero}
\begin{aligned}
\alpha(u) & = s_1 e^{mu} + s_0 \left( q m n  \ln (q n e^{mu} - r)e^{mu} - m (q m n e^{mu} u - r) \right), \\
\eta(v) & = s_2 e^{-nv} + s_0 \left(q e^{-nv} \ln (r m e^{nv} - q) + rm \right), \\
\beta(v) & =  t_1 e^{nv} + t_0 \left( r m n  \ln (r m e^{nv} - q)e^{nv} - n (r m n e^{nv} v - q) \right)  , \\
\kappa(u) & =  t_2 e^{-mu} + t_0 \left(r e^{-mu} \ln (q n e^{mu} - r) + qn \right)    ,
\end{aligned}
\end{equation}
for constants $s_0,s_1,s_2,t_0,t_1,t_2$.  Among these, $s_0, t_0$ are such that both sides of equation \eqref{eqns-for-greeks-nice-mn-1} are equal to $m^2 n r^2 s_0$ and both sides of equation \eqref{eqns-for-greeks-nice-mn-2} are equal to $m n^2 q^2 t_0$.

Finally, consider equations \eqref{fourthline}.  Substituting the expressions \eqref{F1G1-in-terms-of-uv-mn-nonzero} and \eqref{not-so-obvious-greek-solns-mn-nonzero} into \eqref{fourthline} and solving for $\mu$ and $\xi$, we find that $\mu$ is a function of $u$ alone and $\xi$ is a function of $v$ alone if and only if
\[ s_0 = t_0 = s_1 t_2 - s_2 t_1 = 0. \]
Moreover, these conditions imply that $\mu = \xi = 0$.  Thus we see that $\B$ is a transformation between equations of the form
\begin{equation}\label{underlying-pdes-mn-nonzero}
 u_{xy} = s_1 e^{mu} u_x + t_2 e^{-mu} u_y, \qquad v_{xy} = s_2 e^{-nv} v_x + t_1 e^{nv} v_y,
\end{equation}
with $s_1 t_2 = s_2 t_1$.  The transformation is given by
\begin{equation}\label{BT-mn-nonzero}
\begin{aligned}
u_x & = \frac{n(r - qn e^{mu})}{q e^{-nv} - rm} v_x + \frac{1}{r m}(r - q n e^{mu})(t_1 e^{nv} - t_2 ^{-mu}), \\
v_y & = \frac{m(q - rm e^{nv})}{r e^{-mu} - qn} v_x + \frac{1}{q n}(q - r m e^{nv})(s_1 e^{mu} - s_2 e^{-nv}).
\end{aligned}
\end{equation}
Note that this is actually a 1-parameter family of transformations, with parameter $\lambda = q/r$.

\medskip
\noindent
{\em Case 3(a)(ii):} Suppose that one of $m, n$ is equal to zero; without loss of generality, assume that $m \neq 0$ and $n=0$.
Again integrating the second-order ODEs for $F_1, G_1$ gives
\begin{equation}\label{F1G1-in-terms-of-hk-n-zero}
 F_1(h) = \frac{1}{q - m h}, \qquad G_1(k) = r e^{m k}
\end{equation}
for some nonzero constants $q, r \in \R$.  Then we can integrate equations \eqref{duv to dhk} to obtain the local coordinate transformation
\begin{equation}\label{hk-to-uv-nice-mn-n-zero}
u = k - \frac{1}{m} \ln(m h - q), \qquad
v = h + \frac{r}{m} e^{mk}.
\end{equation}
The inverse transformation is given by
\begin{equation}\label{uv-to-hk-nice-n-zero}
h =\dfrac{m v + q r e^{mu}}{m(r e^{mu} + 1)} , \qquad
k = u + \frac{1}{m} \ln \left( \dfrac{m v - q}{r e^{mu} + 1} \right) ,
\end{equation}
and thus we may write the expressions \eqref{F1G1-in-terms-of-hk-n-zero} in terms of the $(u,v)$ coordinates as
\begin{equation}\label{F1G1-in-terms-of-uv-n-zero}
F_1 = \frac{r e^{mu} + 1}{q - m v}, \qquad
G_1 = \frac{r e^{mu} (mv - q)}{r e^{mu} + 1}.
\end{equation}
Now, after clearing denominators appropriately, equations \eqref{zero3ap}, \eqref{zero4ap} may be written as
\begin{align}
(r e^{mu} + 1)(\alpha_u - m \alpha) & = (q - mv )\eta_v, \label{eqns-for-greeks-n-zero-1} \\
r (m v - q)\beta_v  & = (e^{-mu} + r)(\kappa_u + m \kappa).  \label{eqns-for-greeks-n-zero-2}
\end{align}
As in the previous case, both sides of each of these equations must be constant.  A straightforward computation shows that the general solution of equations \eqref{eqns-for-greeks-n-zero-1} and \eqref{eqns-for-greeks-n-zero-2} is given by
\begin{equation}\label{not-so-obvious-greek-solns-n-zero}
\begin{aligned}
\alpha(u) & = s_1 e^{mu} + s_0 \left( r e^{mu} (\ln (r e^{mu} + 1) - m u) - 1 \right), \\
\eta(v) & = s_2  - s_0 \ln (m v - q), \\
\beta(v) & =  t_1  - t_0 r \ln (m v - q)  , \\
\kappa(u) & =  t_2 e^{-mu} + t_0 \left( e^{-mu} \ln (r e^{mu} + 1) - r \right)    ,
\end{aligned}
\end{equation}
where $s_0, t_0$ are such that that both sides of equation \eqref{eqns-for-greeks-n-zero-1} are equal to $m s_0$ and both sides of equation \eqref{eqns-for-greeks-n-zero-2} are equal to $m r^2 t_0$.

Now consider equations \eqref{fourthline}.  Substituting the expressions \eqref{F1G1-in-terms-of-uv-n-zero} and \eqref{not-so-obvious-greek-solns-n-zero} into \eqref{fourthline} and solving for $\mu$ and $\xi$, we find that $\mu$ is a function of $u$ alone and $\xi$ is a function of $v$ alone if and only if $s_0 = t_0 =  0$.
Moreover, these conditions imply that $\mu =  0$ and
\[ \xi(v) =  \frac{1}{m}(s_1 t_2 - s_2 t_1) (mv - q)   .  \]
Thus we see that $\B$ is a transformation between the equation
\[ u_{xy} = s_1 e^{mu} u_x + t_2 e^{-mu} u_y \]
and the constant-coefficient linear equation
\[ v_{xy} = s_2 v_x + t_1 v_y + \frac{1}{m}(s_1 t_2 - s_2 t_1) (mv - q). \]
The transformation is given by
\begin{equation}\label{BT-n-zero}
\begin{aligned}
u_x & = \frac{r e^{mu} + 1}{q - m v} v_x + \frac{1}{m} e^{-mu} (r e^{mu} + 1)(t_1 e^{mu} - t_2), \\
v_y & = \frac{r e^{mu} (mv - q)}{r e^{mu} + 1} v_x + \frac{1}{m}(q - m v)(s_1 e^{mu} - s_2 ).
\end{aligned}
\end{equation}
Note that this is actually a 2-parameter family of transformations, with parameters $q, r$.

If the constants $m_1, n_1, m_2, n_2$ do not satisfy the hypotheses of Lemma \ref{more-solutions-lemma}, then $\alpha, \beta,\kappa,\eta$ are
given by \eqref{obvious-greek-solns}.
However, the remaining analysis is more complicated.  We can still solve ODEs for the functions $F_1$ and $G_1$, as in Lemma \ref{A0B0C0D0-constant-lemma}; for instance, if $m_1, n_1, m_2, n_2$ are all nonzero, then we have
\begin{equation}\label{F1G1-bad-constants}
F_1 = \frac{n_1}{q_1 e^{n_1 h} + m_1}, \qquad G_1 = \frac{n_2}{q_2 e^{n_2 k} + m_2}
\end{equation}
for some nonzero constants $q_1, q_2 \in \R$.  We can still substitute these expressions into equations \eqref{duv to dhk} and integrate to obtain the local coordinate transformation
\begin{equation}\label{hk-to-uv-bad-constants}
u = \frac{1}{m_1} \left( m_1 k + n_1 h - \ln(q_1 e^{n_1 h} + m_1) \right), \qquad v = \frac{1}{m_2} \left( m_2 h + n_2 k - \ln(q_2 e^{n_2 k} + m_2) \right).
\end{equation}
However, the inverse transformation is surprisingly unwieldy, so in order to investigate equations \eqref{fourthline}, we have little choice but to work in $(h,k)$ coordinates and use the transformation \eqref{hk-to-uv-bad-constants} to express the functions \eqref{obvious-greek-solns} in terms of $h$ and $k$.  When we substitute the resulting expressions (along with \eqref{F1G1-bad-constants}) into \eqref{fourthline}, solve for $\mu$ and $\xi$, and impose the conditions that
\[ \Delta \dfrac{\partial \mu}{\partial v} = \left( \dfrac{\di}{\di h} - F_1 \dfrac{\di}{\di k}\right)\mu = 0, \qquad
\Delta \dfrac{\partial \xi}{\partial u} = \left( \dfrac{\di}{\di k} - G_1 \dfrac{\di}{\di h}\right)\xi = 0,
 \]
we find (after substantial computation) that there are no solutions unless $m_1 = -n_2$ and $m_2 = -n_1$.  Thus we conclude that the only solutions for which $A_0, B_0, C_0, D_0$ are all constant are those described above.
\end{proof}

Before leaving Case 3(a), we wish to show that there are weaker hypotheses that lead to the classification in Prop. \ref{F1G1-independent-A0B0C0D0-const-prop}.  To see how these arise, we begin by solving the equations \eqref{thirdline}
by integration.  Using \eqref{fogovals} we can rewrite these equations as
\begin{align}
\dfrac{\di}{\di h} \left( (\alpha-\eta) \dfrac{F_1}{\Fp} \right) &= \eta, \label{zero3app}\\
\dfrac{\di}{\di k} \left( (\beta-\kappa) \dfrac{G_1}{\Gp} \right) &= \kappa. \label{zero4app}
\end{align}
(Note that these equations are equivalent to \eqref{zero3ap} and \eqref{zero4ap}.)
Let $\teta$ be a $v$-antiderivative of $\eta$; then, because $\di \teta/\di h = \eta$, we can integrate \eqref{zero3app} to get
\begin{equation} (\alpha-\eta) \dfrac{F_1}{\Fp} = \teta + \psi_0,\label{zero3a} \end{equation}
where $\psi_0$ is a some function of $k$ only. Next, letting $\talpha$ denote a $u$-antiderivative of $\alpha$ allows us to
rewrite \eqref{zero3a} as
$$\dfrac{\di}{\di h} \talpha = F_1 \eta + \Fp (\teta + \psi_0),$$
and integrating this with respect to $h$ (holding $k$ fixed) gives
\begin{equation}\label{zero3b}
\talpha = F_1 (\teta + \psi_0) + \psi_1,
\end{equation}
where $\psi_1$ is some function of $k$.  Applying $\di/\di_k$ to both sides gives
\begin{equation}\label{zero3c}
\alpha = F_1 G_1 \eta + F_1 \psi_0' + \psi_1'.
\end{equation}
We now use \eqref{zero3a} and \eqref{zero3c} to solve for 
$\eta$, yielding
\begin{equation}
\eta = \dfrac1\Delta\left(\psi_1' + F_1 \psi_0' - \dfrac{\Fp}{F_1} (\teta + \psi_0)\right).\label{zero3e}
\end{equation}
Similarly, letting $\tbeta$ denote a $v$-antiderivative of $\beta$ and $\tkappa$ a $u$-antiderivatives of $\kappa$, we have
\begin{equation}\label{zero4ab}
(\beta-\kappa) \dfrac{G_1}{\Gp} = \tkappa+\phi_0, \quad \tbeta = G_1 (\tkappa+\phi_0) + \phi_1
\end{equation}
for some functions $\phi_0, \phi_1$ of $h$ only, and by differentiating the latter equation and eliminating $\beta$ we obtain
\begin{equation}\label{zero4e}
\kappa = \dfrac1\Delta\left(\phi_1' + G_1 \phi_0' -\dfrac{\Gp}{G_1}(\tkappa+\phi_0)\right).
\end{equation}

The left-hand sides of \eqref{zero3e},\eqref{zero4e} are functions of  $v$ only (respectively, $u$ only), so applying the
operator $\Delta \di/\di_u = \di/\di_k - G_1 \di/\di_h$ (resp., $\Delta \di/\di_v$) yields zero on the left, while on the right we can obtain an expression that is linear in $\teta$ with coefficients involving $F_1, G_1, \psi_0, \psi_1$ and their derivatives
(resp., linear in $\tkappa$ with coefficients involving $F_1, G_1, \phi_0, \phi_1$ and derivatives).  The resulting equations are
\begin{align}
\dfrac{R}{F_1^2\Delta^2}(\teta + \psi_0) - \dfrac{\di}{\di k} \left( \dfrac{\psi_1' + F_1 \psi_0'}{\Delta}\right)+\dfrac{(\psi_0' + F_1 G_1^2 \psi_1')\Fp}{F_1 \Delta^2} &= 0 \label{zero3f} \\
\dfrac{S}{G_1^2 \Delta^2} (\tkappa + \phi_0) - \dfrac{\di}{\di h} \left( \dfrac{\phi_1' + G_1 \phi_0'}{\Delta}\right) + \dfrac{(\phi_0' + F_1^2 G_1 \phi_1')\Gp}{G_1 \Delta^2} &= 0 \label{zero4f},
\end{align}
where the coefficients $R, S$ (which will play an important
role in what follows) are given by
$$
R = G_1 \Delta( \Fp^2- F_1 \Fpp) + (F_1 \Gp - G_1^2 \Fp) F_1 \Fp, \quad
S = F_1 \Delta(\Gp^2 - G_1 \Gpp) + (G_1 \Fp - G_1^2 \Gp) G_1 \Gp.
$$
Note that these expressions are symmetric under the interchange of $F_1$ and $G_1$.

\begin{prop}\label{tfaq}  The following conditions are equivalent:
\renewcommand{\theenumi}{\roman{enumi}}%
\begin{enumerate}
\item $R=0$; \label{lone}
\item $F_1$ is a product of a function of $u$ and a function of $v$; \label{ltwo}
\item $S=0$; \label{lthree}
\item $G_1$ is a product of a function of $u$ and a function of $v$; \label{lfour}
\end{enumerate}
\end{prop}
\begin{proof} To show that \eqref{lone} and \eqref{ltwo} are equivalent, we use the identity
\begin{align*}
\dfrac{R}{F_1^2} &= -G_1 \Delta\dfrac{\di}{\di h} \dfrac{\Fp}{F_1} + \dfrac{\Fp}{F_1} \left( \dfrac{\di}{\di k} - G_1 \dfrac{\di}{\di h}\right) (F_1 G_1) \\
&= \Delta^2 \dfrac{\di}{\di u}\dfrac{\Fp}{F_1}  -  \dfrac{\Fp}{F_1} \Delta \dfrac{\di}{\di u} \Delta.
\end{align*}
Thus, if $R=0$ then $\dfrac{\di}{\di u} \ln\left(\dfrac{\Fp}{\Delta F_1}\right)=0$, and hence $\Fp/F_1 = a \Delta $ for some function $a(v)$.
It would follow that
$$\dfrac{\di}{\di v} \ln F_1 = \dfrac{1}{\Delta}\dfrac{\di}{\di h} \ln F_1 = a,$$
and thus $\dfrac{\di^2}{\di u \di v} \ln F_1=0$, giving (ii).  Similarly, if $F_1 = A(v) B(u)$, then it is straightforward to compute
that $R$ vanishes.  By symmetry, \eqref{lthree} and \eqref{lfour} are equivalent.

To show that \eqref{lone} and \eqref{lthree} are equivalent, we will need two further identities.  The first is
\begin{equation}\label{Rident}
\dfrac{R}{\Delta G_1 F_1^2 \Fp} = \dfrac{\Fp}{F_1^2} - \dfrac{\Fpp}{F_1 \Fp} + \dfrac{\di}{\di u} \ln(F_1 G_1).
\end{equation}
This follows from writing
\begin{equation}\label{Rident2}
\dfrac{R}{G_1 F_1^2 \Fp} = \left( \dfrac{\Fp}{F_1^2} -\dfrac{\Fpp}{F_1 \Fp}\right) \Delta + \dfrac{\Gp}{G_1} - G_1 \dfrac{\Fp}{F_1},
\end{equation}
and interpreting the last two terms as the result of applying $\Delta \dfrac{\di}{\di u} = \dfrac{\di}{\di k} - G_1 \dfrac{\di}{\di h}$ to $\ln(F_1 G_1)$.  The second identity is
\begin{equation}\label{Sident}
\dfrac{S}{\Delta^2 F_1 G_1^2} = -\dibk \dibu \ln(F_1 G_1).
\end{equation}
This follows from computing
$$\left[ \dibk, \dibu \right] = \left[ \dibu + G_1 \dibv, \dibu \right] = -G_{1,u} \dibv = -\dfrac{\Gp}{\Delta^2} \left( \dibh - F_1 \dibk\right)$$
and expanding
\begin{align*}\dibk \dibu \ln(F_1 G_1)  &= -\dfrac{\Gp}{\Delta^2} \left( \dibh - F_1 \dibk\right)\ln(F_1 G_1) + \dibu \dibk \ln G_1 \\
&= -\dfrac{\Gp}{\Delta^2}\left( \dfrac{\Fp}{F_1} - \dfrac{F_1 \Gp}{G_1}\right) + \dfrac{1}\Delta \dibk \left( \dfrac{\Gp}{G_1}\right).
\end{align*}
If $R=0$, then applying $\dibk$ to \eqref{Rident} implies that the right-hand side of \eqref{Sident} vanishes, and thus $S=0$.
Similarly, $S=0$ implies $R=0$ by symmetry.
\end{proof}

\begin{cor} The conditions in Prop. \ref{tfaq} imply the hypotheses of Prop. \ref{F1G1-independent-A0B0C0D0-const-prop}.
\end{cor}

\begin{proof}
Assuming $R=0$, \eqref{Rident} implies that
$$A_0 = \dfrac{\Fp}{F_1^2}-\dfrac{\Fpp}{F_1 \Fp}= -\dfrac{\di}{\di u} \ln(F_1 G_1).$$
The left-hand side is a function of $h$ while, by conditions (ii) and (iv) above, the right-hand side is a function of $u$.
Thus, both sides are constant.  By symmetry, $S=0$ implies that $C_0$ is constant.
\end{proof}

It turns out that the conditions in Prop. \ref{tfaq} also hold when the right-hand sides of the underlying PDEs for $u$ and $v$
are homogeneous in the first-order partials.  (Of course, the converse is not true, the exception being the transformation given by \eqref{BT-n-zero} above.)

\begin{prop}  If $\mu=\xi=0$, then the conditions in Prop. \eqref{tfaq} hold.
\end{prop}
\begin{proof}
Solving \eqref{fourthline} for the Greek-letter variables gives
\begin{equation}\label{solve42d}
\begin{pmatrix} \mu +\alpha F_0 \\ \xi + \beta G_0 \end{pmatrix} = \dfrac{1}{\Delta}\begin{pmatrix} G_0 F_{0,v} + F_1 F_0 G_{0,u} \\ F_0 G_{0,u} + G_1 G_0 F_{0,v}\end{pmatrix}.
\end{equation}
Using \eqref{thirdline} to substitute for $G_{0,u}$ in the top equation in \eqref{solve42d}
and for $F_{0,v}$ in the bottom equation, and then solving for $\mu,\xi$ gives
\begin{align}\mu &= \dfrac{1}{\Delta F_0^2} \left( \dfrac{\eta - \alpha}{F_0} - \dibv \left( \dfrac{G_0}{F_0}\right) \right)
= -\dfrac{F_0 G_0}{\Delta}  \dibv \ln\left( \dfrac{F_1 G_0}{F_0}\right) \label{case31mu} \\
\xi &= \dfrac{1}{\Delta G_0^2} \left( \dfrac{\kappa - \beta}{G_0} - \dibu \left( \dfrac{F_0}{G_0}\right) \right)
= -\dfrac{F_0 G_0}{\Delta}  \dibu \ln\left( \dfrac{G_1 F_0}{G_0}\right) \label{case31xi}
,
\end{align}
where the expressions on the right are obtained by using \eqref{fogovals} to eliminate $\eta -\alpha$ and $\kappa-\beta$.

Thus, if $\mu=\xi=0$ then $F_1 G_0 = \sigma F_0$ and $G_1 F_0 = \tau G_0$ for some functions $\sigma, \tau$ of $u,v$ respectively.
It follows that
$$0 = \dibh \dibk \ln(F_1 G_1) = \dibh \left(\dfrac{\sigma'}{\sigma}\right) + \dibk \left(\dfrac{\tau'}{\tau}\right) = F_1 \dibu \left(\dfrac{\sigma'}{\sigma}\right) + G_1 \dibv \left(\dfrac{\tau'}{\tau}\right).$$
If both derivatives of $\sigma'/\sigma$ and $\tau'/\tau$ are zero, then $\sigma,\tau$ are exponential functions of their arguments, and
substituting for $F_1 G_1 = \sigma \tau$ in \eqref{Sident} gives $S=0$.  If these derivatives are not zero, then multiplying the last displayed
equation by $F_1$ gives
$$F_1^2\dibu \left(\dfrac{\sigma'}{\sigma}\right)  = -\sigma \tau \dibv \left(\dfrac{\tau'}{\tau}\right).$$
Thus, $F_1$ is a function of $u$ times a function of $v$.
%
\end{proof}

{\bf Case 3(b):} Suppose that at least one of the pairs $(A_0, B_0)$, $(C_0, D_0)$ consists of non-constant functions.  Without loss of generality, suppose that $A_0$ and $B_0$ are non-constant functions of $h$, and consider equation \eqref{zero3ap}.  The derivatives of this equation with respect to $h$ and $k$ (using \eqref{deehkdefs} to compute derivatives of $\alpha$ and $\eta$) yield two equations that can be solved for $\alpha_{uu}$ and $\eta_{vv}$ to obtain
\begin{equation}\label{messy-second-derivatives}
\begin{aligned}
\alpha_{uu} & = \frac{1}{1 - F_1 G_1} \left((G_1 B_0 - A_0) \alpha_u  + G_1 (A_0' + A_0 B_0 - F_1 A_0^2)\alpha - G_1 A_0' \eta   \right), \\
\eta_{vv} & = \frac{1}{1 - F_1 G_1} \left( F_1(B_0 - F_1 A_0) \alpha_u - (1 - F_1 G_1) B_0 \eta_v + F_1(A_0' + A_0 B_0 - F_1 A_0^2) \alpha - F_1 A_0' \eta \right).
\end{aligned}
\end{equation}
Then computing
\[ (\alpha_{uu})_v = (\eta_{vv})_u = 0 \]
and taking \eqref{messy-second-derivatives} into account yields a relation between $\alpha_u$, $\eta_v$, $\alpha$, and $\eta$, which together with \eqref{zero3ap} can be solved for $\alpha_u$ and $\eta_v$.
Repeating the process, computing
\[ (\alpha_{u})_v = (\eta_{v})_u = 0 \]
yields two independent linear relations of the form
\[ Y_i(h,k) \alpha + Z_i(h,k) \eta = 0, \qquad i=1, 2 \]
where $Y_i$ and $Z_i$ are long differential polynomials in $F_1$ and $G_1$.
This leads to several conditions that must be satisfied:
\begin{enumerate}
\item In order for there to exist nonzero solutions $(\alpha, \eta)$, we must have $Y_1 Z_2 - Y_2 Z_1 = 0$.  This condition yields a complicated polynomial equation that is simultaneously a (highly nonlinear) $5$th-order ODE for $F_1$ and a $2$nd-order ODE for $G_1$.
\item Assuming that the previous condition is satisfied, each of the ratios $\frac{Y_i}{Z_i} = -\frac{\eta}{\alpha}$ must be equal to a function of $u$ alone times a function of $v$ alone, and so must satisfy the PDE
\[ \dfrac{\di^2}{\di u \di v}\left( \ln \left( \frac{Y_i}{Z_i} \right) \right) = 0. \]
This condition leads to additional, higher-order equations that must be satisfied by $F_1$ and $G_1$.
\end{enumerate}
We suspect that these conditions are simply too overdetermined and that there are no solutions to \eqref{zero3ap} with $A_0$ and $B_0$ non-constant---let alone that any such solutions might also satisfy equations \eqref{fourthline}.  Unfortunately the algebra is too complicated to carry out to completion, so we must leave this conjecture unsettled for the time being.

We summarize the results of Propositions \ref{F1G1-constant-prop}, \ref{F1G1-dependent-prop}, and \ref{F1G1-independent-A0B0C0D0-const-prop} in the following theorem (also taking note of Theorem \ref{finite-dimension-thm}):

\begin{theorem}\label{all-the-BTs-theorem}
Any quasilinear, wavelike, autonomous normal B\"acklund transformation is one of the following:
\begin{itemize}
\item If $F_1, G_1$ are constant, then:
   \begin{enumerate}
   \item If $F_1 G_1 \neq -1$, then $\B$ is a transformation between two constant-coefficient linear PDEs.
   \item If $F_1 G_1 = -1$, then $\B$ is one of the following:
       \begin{enumerate}
       \item a transformation between two constant-coefficient linear PDEs;
       \item a transformation between two PDEs of the form
\[ u_{xy} = a_1 e^{2ru} + a_2 e^{-2ru}, \qquad v_{xy} = b_1 e^{2rv} + b_2 e^{-2rv}; \]
        \item a transformation between two PDEs of the form
\[ u_{xy} = a_1 \cos(2ru) + a_2 \sin(2ru), \qquad v_{xy} = b_1 \cos(2rv) + b_2 \sin(2rv). \]
       \end{enumerate}
    \end{enumerate}
\item If $F_1, G_1$ are functionally dependent but not both constant, then one of them, say $G_1$, is constant, and  $\B$ is one of the following:
       \begin{enumerate}
       \item a transformation between two constant-coefficient linear PDEs;
       \item a transformation between either two PDEs of the form
\[ u_{xy} = c_3 e^{G_1 r_2 u}, \qquad v_{xy} = c_2 e^{r_2 v} v_y, \]
or two PDEs of the form
\[ u_{xy} = c_1 e^{-G_1 r_2 u} u_x, \qquad v_{xy} = c_4 e^{-r_2 v}; \]
        \item a transformation between two PDEs of the form
\[ u_{xy} = c_{1} e^{G_1 r_0 u} u_x, \qquad v_{xy} = c_{2} e^{r_0 v} v_y; \]
        \item a transformation between two PDEs of the form
\[ u_{xy} = c_1 e^{G_1 r_0 u} \left( G_1 u_x - \frac{b}{r_0} \right) + b u_y , \qquad
v_{xy} = c_2 e^{r_0 v} \left( v_y - \frac{b}{r_0} \right) + b v_x
\]
with $b, c_1, c_2 \neq 0$.
        \end{enumerate}
\item If $F_1, G_1$ are functionally independent and the functions $A_0, B_0, C_0, D_0$
defined by \eqref{A0B0C0D0-expressions})
are all constant, then $\B$ is one of the following:
       \begin{enumerate}
       \item a transformation between two PDEs of the form
\[   u_{xy} = s_1 e^{mu} u_x + t_2 e^{-mu} u_y, \qquad v_{xy} = s_2 e^{-nv} v_x + t_1 e^{nv} v_y   \]
with $m, n \neq 0$ and $s_1 t_2 = s_2 t_1$;
        \item a transformation between a PDE of the form
\[  u_{xy} = s_1 e^{mu} u_x + t_2 e^{-mu} u_y   \]
with $m \neq 0$ and a constant-coefficient linear PDE.
        \end{enumerate}
\item A possible finite-dimensional family of transformations with $F_1, G_1$ functionally independent and where
at least one of the pairs $(A_0, B_0)$, $(C_0, D_0)$.
\end{itemize}
\end{theorem}

\section{Discussion}\label{discussion-sec}

We close by noting that not all interesting \BT/s involving
integrable PDEs fit the restrictions we put in place at the beginning of \S\ref{BTdef}; in particular, the total
$\B$ can have arbitrary dimension.
For example, consider the Tzitzeica equation
\begin{equation}\label{tzitz}
(\ln h)_{xy} = h - h^{-2},
\end{equation}
which arises in connection with the construction of affine spheres \cite{RS}.
(Note that this PDE can be put into the wavelike form \eqref{waveform}
by setting $u=\ln h$.)  This equation has an auto-\BT/ which
is defined by a compatible system of total differential equations
\begin{align*}
\alpha_x &= \dfrac{ h_x\alpha + \lambda \beta}{h} -\alpha^2, &
\alpha_y &= h - \alpha \beta, \\
\beta_x &= h-\alpha \beta, & \beta_y &= \dfrac{ h_y \beta + \lambda^{-1}\alpha}{h},
\end{align*}
where $\lambda$ is an arbitrary nonzero constant.  Given a solution $h$ of \eqref{tzitz},
one solves for $\alpha$ and $\beta$, and then the new solution of \eqref{tzitz} is given by $h' = 2\alpha \beta - h$.
(Note that the system is symmetric under the interchanging of $h$ and $h'$.)
The system for $\alpha$ and $\beta$ is equivalent to a Pfaffian system of rank 2 defined
on a 7-dimensional total space $\B$, with submersions to 5-manifolds $\M, \Mbar$, each carrying
a copy of the \MA/ system encoding \eqref{tzitz}.

To give another example, in our previous paper \cite{CI09} we proved that
(almost) all hyperbolic \MA/ equations that are Darboux-integrable at second order
are linked by a \BT/ to the wave equation $z_{xy} =0$, and this transformation
is of the type described in \S\ref{BTdef}.
These transformations may be composed, in an obvious way, to yield a more general \BT/
between any two of these equations (for example, $u_{xy} = 2\sqrt{u_x u_y}/(x+y)$
and $v_{xy} =2v/(x+y)^2$) where, again, the total space has dimension 7.  However,
it also happens that certain pairs of these equations are linked to each other
by quasilinear wavelike \BT/s of the type discussed in this paper, where the total
space is 6-dimensional, and which do not appear to involve the wave equation.

\end{document}